\documentclass[a4paper, reqno]{amsart}
\usepackage{amssymb}
\usepackage{amsmath}
\usepackage{amsthm}
\usepackage[utf8]{inputenc}
\usepackage[T1]{fontenc}
\usepackage{lmodern}
\usepackage{ae}
\usepackage{graphicx}
\usepackage[displaymath,modulo,mathlines]{lineno}
\usepackage{enumerate}
\usepackage{fancybox}
\usepackage{setspace}
\usepackage[all,xdvi]{xy}
\usepackage{ifthen}
\usepackage{tikz}
\usetikzlibrary{shapes,matrix,arrows}

\modulolinenumbers[2]
\numberwithin{equation}{section}
\renewcommand{\det}{\operatorname{det}}
\newcommand{\RR}{\mathbb{R}}
\newcommand{\C}{\mathbb{C}}
\newcommand{\N}{\mathbb{N}}
\newcommand{\Z}{\mathbb{Z}}
\newcommand{\Q}{\mathbb{Q}}

\newcommand{\cg}{\mathcal{G}}

\newcommand{\mf}[1]{\mathfrak{#1}}

\newcommand{\tth}{t_{\mf h}}

\newcommand{\kz}{\textsc{kz}}
\newcommand{\alg}{alg}

\newcommand{\id}{\operatorname{id}}

\newcommand{\im}{\operatorname{Im}}
\newcommand{\csch}{\operatorname{csch}}
\newcommand{\R}{\mathcal{R}}
\newcommand{\Ad}{\operatorname{Ad}}
\newcommand{\ul}{\underline}
\newcommand{\cyb}{\operatorname{CYB}}
\newcommand{\bcyb}{\overline{\operatorname{CYB}}}
\newcommand{\alt}{\operatorname{Alt}}
\newcommand{\dd}{\operatorname{d}}
\newcommand{\Sh}{\operatorname{Sh}}
\newcommand{\sh}{\operatorname{sh}}

\newcommand{\h}{\hslash}

\newcommand{\ot}{\otimes}
\newcommand{\hh}{\mf h}
\newcommand{\mm}{\mf m}

\newcommand{\tsh}{\tilde\sigma_{\h}}
\newcommand{\ts}{\tilde\sigma}

\newcommand{\modu}{\operatorname{-mod}}

\newtheorem{thm}{Theorem}[section]
\newtheorem{cor}[thm]{Corollary}
\newtheorem{prop}[thm]{Proposition}
\newtheorem{lem}[thm]{Lemma}

\newtheorem{defi}[thm]{Definition}

\theoremstyle{remark} 
\newtheorem{rmk}[thm]{Remark}

\author{Adrien Brochier}
\address{IRMA (CNRS), rue Ren\'e Descartes, F-67084 Strasbourg, FRANCE}
\email{brochier@math.unistra.fr}
\title[A KD theorem for cyclotomic KZ connections]{A Kohno--Drinfeld theorem for the monodromy of cyclotomic KZ connections}

\date{\today}

\begin{document}

\begin{abstract}
We compute explicitly the monodromy representations of ``cyclotomic'' analogs of the Knizhnik--Zamolodchikov differential system. These are representations of the type B braid group $B_n^1$. We show how the representations of the braid group $B_n$ obtained using quantum groups and universal $R$-matrices may be enhanced to representations of $B_n^1$ using dynamical twists. Then, we show how these ``algebraic'' representations may be identified with the above ``analytic'' monodromy representations. 
\end{abstract}
\maketitle
\setcounter{section}{-1}
\tableofcontents
\section{Introduction}
The Knizhnik--Zamolodchikov (KZ) differential system arises in the study of correlation functions in conformal field theory. The study of its monodromy was one of the first motivations for Drinfeld's theory of quantum groups and associators. To each pair of a simple Lie algebra $\mf g$ and a finite dimensional $\mf g$-module $V$, one attaches a KZ differential system. Its monodromy leads to an analytic representation of the Artin braid group $B_n$~\cite{Artin1925} in $V^{\ot n}[[\h]]$. On the other hand, the theory of quantum groups also leads to representations of $B_n$. More precisely, every pair of an algebra $A$ and an element $R \in (A^{\ot 2})^{\times}$ which is solution of the quantum Yang-Baxter Equation (QYBE)
\begin{equation}\label{eq:QYBE}
 R^{1,2}R^{1,3}R^{2,3}=R^{2,3}R^{1,3}R^{1,2}
\end{equation}
allows one to construct representations of $B_n$ in $\tilde V^{\ot n}$ for each $\tilde V \in A\modu$\footnote{If $A$ is an (Lie) algebra over $\C$, then we denote by $A\modu$ the category of finite dimensional $A\modu$ules. If $A$ is an algebra over $\C[[\h]]$, then $A\modu$ is the category of $A$-modules which are topologically free $\C[[\h]]$-modules of finite rank.} by:
\[
 \sigma_i \longmapsto \left( (i,i+1)R^{i,i+1} \right)_{\vert \tilde V^{\ot n}}
\]
where $\sigma_1,\dots,\sigma_{n-1}$ are the Artin generators and $(i,i+1)\in S_n$ acts by permuting the tensor factors of $\tilde V^{\ot n}$. In particular, every quasi-triangular bialgebra (QTBA) leads to a solution of the QYBE.  To $\mf g$ is attached a QTBA $(U_{\h}(\mf g),\R_{\h})$ in the category of free $\C[[\h]]$-modules. It was proved by V.~Drinfeld that there is an algebra isomorphism $U(\mf g)[[\h]] \cong U_{\h}(\mf g)$, and therefore an equivalence of categories between $\mf g\modu$ and $U_{\h}(\mf g)\modu$, denoted by $V\mapsto V_{\h}$. As a $\C[[\h]]$-module, $V_{\h}\cong V[[\h]]$. Thus, for any $V \in \mf g\modu$, we obtain an algebraic representation of $B_n$ in $V^{\ot n}[[\h]]$.

The Kohno-Drinfeld theorem~\cite{Drinfeld1990a,Drinfeld1990,Kohno1987b} asserts that the ``analytic'' and the ``algebraic'' representations of $B_n$ in $V^{\ot n}[[\h]]$ are actually equivalent.

We attach a ``cyclotomic'' analog of the KZ differential system to the data of: a simple Lie algebra $\mf g$, an integer $N\geq 2$, an automorphism $\sigma$ of $\mf g$ such that $\sigma^N=\id_{\mf g}$ and $\mf g^{\sigma}$ is a Cartan subalgebra $\mf h$, a $\mf g$-module $V$ and a $\mf h$-module $W$. It leads to a monodromy representation of the type B braid group $B_n^1$ in $W \ot V^{\ot n}[[\h]]$. On the other hand, we show that the representations of $B_n$ coming from a pair $(A,R)$ can be enhanced to representations of $B_n^1$: from the data of a subalgebra $C \subset A$ and elements $E,K \in (C\ot A)^{\times}$ satisfying suitable axioms, the above representation of $B_n$ is extended by setting
\[
 \tau \longmapsto \left ( \prod_{i=2}^{n} (R^{1,i})^{-1} \prod_{i=2}^n K^{1,i} E^{0,1} \right)_{\vert \tilde W \ot \tilde V^{\ot n}}
\]
where $\tau$ is the additional Artin generator of $B_n^1$ and $\tilde W \in C\modu$. Such data can be constructed as follows (see Section~\ref{sec:state}): $(A,R)=(U_{\h}(\mf g)\rtimes \Z,\R_{\h})$, $C=U_{\h}(\mf h)$, 

\begin{align}\label{eq:alg-rep}
 K&=e^{\h t_{\mf h} /2}& E&=e^{\h(t_{\mf h}+ \frac12 t_{\mf h}^{2,2})}(1\ot \tilde \sigma_{\h})
\end{align}
 
The Drinfeld isomorphism induces an algebra isomorphism $U_{\h}(\mf h) \cong U(\mf h)[[\h]]$, and therefore an equivalence of categories between $\mf h\modu$ and $U_{\h}(\mf h)\modu$, denoted by $W \mapsto W_{\h}$. Again, $W_{\h}\cong W[[\h]]$ as a $\C[[\h]]$-module. Thus, any pair $(W,V) \in \mf h\modu\times \mf g\modu$ gives rise to a representation of $B_n^1$ in $W \ot V^{\ot n}[[\h]]$.

The main goal of this paper is to prove the following generalization of the Kohno--Drinfeld theorem:
\begin{thm}
 These two representations of $B_n^1$ in $W \ot V^{\ot n}[[\h]]$ are equivalent.
\end{thm}

 The QTBA $U_{\h}(\mf g)$ admits a rational form $U_q(\mf g)$ which is defined over the field $\Q(q)$. Although $\R_{\h} \not \in U_q(\mf g)^{\ot 2}$, it still acts in a well defined way on finite dimensional $U_q(\mf g)$-modules, so the above representations of $B_n$ extends to this case. Similarly, the representations of $B_n^1$ which we construct extend to this setup.

Recall the main steps of the proof of the Kohno--Drinfeld theorem in~\cite{Drinfeld1990a,Drinfeld1990}: Drinfeld first introduced the notion of Quasi-Triangular Quasi-Bialgebra (QTQBA). This is a set $(A,\Delta_A,\R_A,\Phi_A)$ where $A$ is an algebra, satisfying suitable axioms. Each $\tilde V \in A\modu$ then gives rise to a representation of $B_n$ in $\tilde V^{\ot n}$. QTQBAs can be modified by an operation called ``twist'', which preserves the equivalence classes of representations of $B_n$.

Using the KZ differential system, one defines an element $\Phi_{\kz} \in U(\mf g)^{\ot 3}[[\h]]$, which allows to construct a QTQBA $(U(\mf g)[[\h]], \Delta_0,\exp(\h t/2), \Phi_{\kz})$ (here, $t\in S^2(\mf g)^{\mf g}$). For $V \in \mf g\modu$, the induced representation of $B_n$ in $V^{\ot n}[[\h]]$ coincides with the representation coming from the monodromy of the KZ connection.

The QTBA $(U_{\h}(\mf g),\Delta_{\h},\R_{\h})$ is seen as a QTQBA by setting $\Phi=1$. Using rigidity arguments, Drinfeld proves that the above two QTQBAs are twist equivalent, and therefore that the representations of $B_n$ that they give rise to are equivalent.

We now give an idea of the proof of our main result. The relation between QTQBAs and representations of $B_n$ can be enhanced as follows: in~\cite{Enriquez2008}, B.~Enriquez defines the notion of a Quasi-Reflection Algebra (QRA) $(B,\Delta_B,E_B,\Psi_B)$ over a QTQBA ($A,\Delta_A,\R_A,\Phi_A)$. These data satisfy in particular the octagon and the mixed pentagon axiom.

The theory of quasi-reflection algebras (QRA) leads to representations of $B_n^1$: if $(B,\Psi,E)$ is a QRA over a QTQBA $(A,\R)$, $\tilde W$ is a $B$-module and $\tilde V$ is an $A$-module, then one can construct a representation of $B_n^1$ in $\tilde W \ot \tilde V^{\ot n}$ which is compatible with the representation of $B_n$ on $\tilde V^{\ot n}$.

Each $(\tilde V,\tilde W) \in A\modu\times B\modu$ then gives rise to a representation of $B_n^1$ in $\tilde W \ot \tilde V^{\ot n}$. As before, one defines a ``twist'' operation for QRAs over QTQBAs, which does not change the equivalence classes of representations of $B_n^1$. To ($\mf g, t,\sigma)$, one then attaches two QRAs:
\begin{enumerate}
 \item $(U_{\h}(\mf h)[[\h]],\Delta_0,E_{\kz},\Psi_{\kz})$ over $(U(\mf g)[[\h]], \Delta_0,\exp(\h t/2), \Phi_{\kz})$ arising from the cyclotomic KZ system. The induced representations coincide with those coming from the monodromy of the KZ connection (Section~\ref{sec:equiv}).
 \item A QRA $(U_{\h}(\mf h),\Delta_{\h},E_{\h},\Psi_{\h})$ constructed as follows: The authors of~\cite{Arnaudon1998,Babelon1991,Buffenoir1999,Etingof2000} construct a dynamical twist, that is an element of $Fun(\mf h^*,U_{\h}(\mf g)^{\ot 2})$ which satisfies the dynamical cocycle equation. It is defined as the solution of a linear equation. On the other hand, the authors of~\cite{Enriquez2007a} construct an algebraic dynamical twist using a quantum analog of the Shapovalov form. It is an element of a suitable localization of $U_{\h}(\mf h)\ot U_{\h}(\mf g)^{\ot 2}$, and the dynamical cocycle equation corresponds to the mixed pentagon equation. We show that it also satisfies the algebraic analog of the ABRR linear equation. We then introduce a shifted (in $\mf h^*$) version $\Psi_{\h}$ of this dynamical twist and show that it actually belongs to the non-localized algebra $U_{\h}(\mf h)\ot U_{\h}(\mf g)^{\ot 2}$ (the shift is related to $\sigma$). We use the fact that $\Psi_{\h}$ satisfies a modified ABRR equation to construct an element $E_{\h}$ such that $(\Psi_{\h},E_{\h})$ satisfies the octagon equation. We show that the corresponding representations of $B_n^1$ are actually given by~\eqref{eq:alg-rep}.
\end{enumerate}
 Finally, using rigidity arguments, we then prove that the two QRAs are twist equivalent. This implies that the corresponding representations of $B_n^1$ are equivalent.

\medskip
\noindent
 {\bf Acknowledgments.} I am very grateful to Benjamin Enriquez for his considerable help during this project. I would also like to acknowledge helpful discussions with Damien Calaque about the classification of dynamical twists.

\subsection{Notation}\label{sec:notation}
Let $\mf g$ be a simple Lie algebra, $\mf h$ be a Cartan subalgebra and $\mf g= \mf n^+ \oplus \mf h \oplus \mf n^-$ a triangular decomposition. Let $G$ be the simply connected complex Lie group with Lie algebra $\mf g$ and $H$ be the subgroup of $G$ with Lie algebra $\mf h$. Pick an element $t \in S^2(\mf g)^{\mf g}$ and let $t_{\mf h} \in S^2(\mf h)$ be its projection onto $S^2(\mf h)$ with respect to the above decomposition. Let $N \geq 2$ be an integer, and choose an automorphism $\sigma$ of $\mf g$ which satisfies:
\begin{enumerate}[(i)]
 \item $\sigma^N=\id_{\mf g}$
 \item $\sigma=\Ad(X)$ for some $X \in H$
 \item $\mf g^{\sigma}=\mf h$
\end{enumerate}
Let $R \subset \mf h^*$ be the set of roots of $\mf g$, $\Pi=\{\alpha_1,\dots,\alpha_r \}$ a choice of simple roots and let $(e_i^{\pm},h_i)_{i=1,\dots , r}$ be the corresponding Chevalley generators. These conditions imply that $\sigma$ is determined by a set of $N$th roots of unity $(\zeta_i)_{i=1,\dots,r}$. Then, $\sigma$ acts by: $\sigma(e_i^{\pm})=\zeta_i^{\pm 1} e_i^{\pm}$, $\sigma(h_i)= h_i$, $i=1,\dots,r$.

Recall~\cite{Brieskorn1971} that the \emph{braid group of Coxeter type B} $B_n^1$ admits the following presentation:
\begin{align*}
B_n^1=\langle \tau, \sigma_1,\dots, \sigma_{n-1}\ |\ & \tau \sigma_1 \tau \sigma_1 = \sigma_1 \tau \sigma_1 \tau \\ 
& \tau \sigma_i=\sigma_i\tau \text{ if } i > 1 \\
 & \sigma_i \sigma_{i+1} \sigma_i = \sigma_{i+1} \sigma_{i} \sigma_{i+1}\ \forall  i \in \{1,\dots,n-2 \}\\
 & \sigma_i \sigma_j = \sigma_j \sigma_i \text{ if } |i-j| \geq 2 \rangle
\end{align*}
The braid group $B_n$ is identified with the subgroup of $B_n^1$ generated by $\sigma_1,\dots,\sigma_{n-1}$.

\section{Statement of the results}\label{sec:state}
In this section, We give a set of axioms which leads to a general algebraic construction of representations of $B_n^1$. We show that such representations can be constructed from the quantized enveloping algebra of any simple Lie algebra. We recall from~\cite{Enriquez2008} the construction of the cyclotomic KZ connections and of the corresponding analytic representations of $B_n^1$. Finally, we state our main result.
\subsection{Algebraic representations of $B_n^1$}

\begin{thm}\label{thm:alg-reps}
Let $A$ be an associative algebra, $C$ a commutative subalgebra of $A$ and elements $E,K \in (C^{\ot 2})^{\times}$, $R \in (A^{\ot 2})^{\times}$ such that:
\begin{enumerate}
 \item $R$ is a solution of the Quantum Yang-Baxter equation~\eqref{eq:QYBE}
 \item $E^{1,2}E^{1,3}(K^{2,3})^2$ commutes with $(2,3)R^{2,3}$
 \item $K^{1,2}K^{1,3}$ commutes with $R^{2,3}$
 \item $K^{1,2}=K^{2,1}$
\end{enumerate}

 There exists a unique group morphism $\rho:B_n^1 \rightarrow (C \ot A^{\ot n}\rtimes S_n)^{\times}$ given by:
 \begin{align*}
  \tau&\longmapsto \prod_{i=2}^n (R^{1,i})^{-1} \prod_{i=2}^n K^{1,i} E^{0,1} \\
  \sigma_i &\longmapsto  (i,i+1)R^{i,i+1}
 \end{align*}

\end{thm}

We also have the following particular case:
\begin{lem}\label{lem:alg}
 If $(A,R)$ is a QTQBA, $C$ is a commutative sub-bialgebra of $A$ and $E,K \in (C^{\ot 2})^{\times}$ satisfy
 \begin{itemize}
  \item $(\id_A \ot \Delta_A)(E)=E^{1,2}E^{1,3}(K^{2,3})^2$
  \item $( \id_A \ot \Delta_A )(K)=K^{1,2}K^{1,3}$
  \item $K^{1,2}=K^{2,1}$
 \end{itemize}
then $(A,C,R,K,E)$ verifies the assumptions of Theorem~\ref{thm:alg-reps}.
\end{lem}

\subsection{Representations of $B_n^1$ attached to $(\mf g,t,\sigma)$}\label{sec:resultRep}
Let $(\mf g, t,\sigma)$ be as in Section~\ref{sec:notation}. Let $(a_{ij})$ be the Cartan matrix of $\mf g$ (that is $a_{ij}=\alpha_j(h_i)$) and let $d_i,\ i=1\dots r$ are the unique coprime integers such that $d_i a_{ij}=d_j a_{ji}$. Recall that $U_{\hbar}(\mf g)$ is the $\C[[\h]]$-algebra topologically generated by $(e_i^{\pm}, h_i)$ for $i=1\dots r$, and relations
\begin{align*}
 [h_i,e_j^{\pm}]&=\pm a_{ij}e_j^{\pm} & [h_i,h_j]&=0 \\
  [e_i^{\pm},e_j^{\mp}]&=\delta_{ij} \frac{e^{\h d_ih_i} - e^{-\h d_ih_i}}{e^{\h d_i}-e^{-\h d_i}}
\end{align*}
and the quantum Serre's relations
\[
 \sum_{k=0}^{1-a_{ij}}\frac{(-1)^k}{[k]_q! [1-a_{ij}-k]_q!} (e_i^{\pm})^{1-a_{ij}-k}e_j^{\pm} (e_i^{\pm})^k=0
\]
where  
\[
[x]_q=\frac{e^{\h x}-e^{-\h x}}{e^{\h}-e^{-\h}} 
\]
and
\[
 [x]_q!=\prod_{i=1}^{x} [i]_q,\ x \in \N
\]

Extend the automorphism $\sigma$ to an automorphism $\sigma_{\h}$ of $U_{\h}(\mf g)$ by setting:
\begin{align*}
 \sigma_{\h}(e_i^{\pm})&=\zeta_i^{\pm}e_i^{\pm} & \sigma_{\h}(h_i)&=0
\end{align*}

Let $U_{\h}(\mf b^{\pm})$  be the subalgebra of $U_{\h}(\mf g)$ generated by $(e_i^{\pm},h_i)_{i=1\dots r}$, and let $U_{\h}(\mf n^{\pm})$ be the subalgebras of $U_{\h}(\mf g)$ generated by $(e_i^{\pm})_{i=1\dots r}$. 
Using a quantum analog of the Weyl group, it is possible to define $e_{\alpha}^{\pm}$ for every $\alpha \in R^+$. Moreover, the triangular decomposition theorem~\cite[Chap. 8]{Chari1994} states that the multiplication defines an isomorphism of $\C[[\h]]$-module from $U_{\h}(\mf n^-) \ot U_{\h}(\mf h) \ot U_{\h}(\mf n^+)$ to $U_{\h}(\mf g)$.
Let $A,B$ be topologically free $\C[[\h]]$-algebras which are flat deformations of $\C$-algebras $A_0$ and $B_0$ respectively. The completed tensor product in the $\h$-adic topology is the $\C[[\h]]$-module
\[
A \hat\ot B = (A_0 \ot B_0)[[\h]]
\]
endowed with the algebra structure induced by those of $A$ and $B$.
$U_{\h}(\mf g)$ can be turned into a topological Hopf algebra, the coproduct
\[
\Delta_{\h}:U_{\h}(\mf g) \rightarrow U_{\h}(\mf g) \hat\ot U_{\h}(\mf g)
\]
being defined by:
\begin{align*}
 \Delta_{\h}(h_i)&= h_i\ot 1 +1 \ot h_i\\
 \Delta_{\h}(e_i^{+})&=e_i^{+} \ot e^{ \h d_ih_i} + 1 \ot e_i^+ & \Delta_{\h}(e_i^{-})&=e_i^{-} \ot 1  + e^{ -\h d_ih_i} \ot e_i^-
\end{align*}
and the antipode
\[
S_{\h}(h_i)=-h_i,\ S_{\h}(e_i^+)=-e_i^+e^{-\h d_i h_i},\ S_{\h}(e_i^-)=-e^{\h d_i h_i}e_i^-
\]
$U_{\h}(\mf g)$ admits a quasi-triangular structure given by an element $\R_{\h} \in U_{\h}(\mf g)^{\hat\ot 2}$. Denotes by $A_{\alg}=(U_{\h}(\mf g)\rtimes \Z$ the semi-direct product in which $\tilde \sigma_{\h} \cong 1 \in \Z$ acts by $\sigma_{\h}$. The coproduct of $U_{\h}(\mf g)$ is extended to $A_{\alg}$ by setting $\Delta_{\h}(\tsh)=\tsh \ot \tsh$. Since $\sigma_{\h}^{\ot 2}(\R_{\h})=\R_{\h}$, $(A_{\alg},\R_{\h})$ is a quasi-triangular Hopf algebra.
\begin{thm}\label{thm:rep_gtsig}
 Let $E_{\h,\sigma}=e^{\h (t_{\mf h}+\frac12 t_{\mf h}^{2,2})}(1\ot \tsh)$ and $K=e^{\h t_{\mf h}/2}$. There exists a unique representation of $B_n^1$ in $(U_{\h}(\mf h)\ot A_{\alg}^{\ot n}\rtimes S_n)^{\times}$ given by:
 \begin{align*}
  \tau&\longmapsto \prod_{i=2}^n (\R_{\h}^{1,i})^{-1}\prod_{i=2}^n K^{1,i} E_{\h,\sigma}^{0,1}  \\
  \sigma_i &\longmapsto  (i,i+1)\R_{\h}^{i,i+1}
 \end{align*}
\end{thm}
\begin{proof}
 According to Lemma~\ref{lem:alg}, it is enough to prove that
\begin{equation}\label{eq:temp}
  (E_{\h,\sigma})^{1,23}= (E_{\h,\sigma})^{1,2}(E_{\h,\sigma})^{1,3}(K^{2,3})^2
 \end{equation}
 Indeed:
 \begin{align*}
  (1 \ot \Delta_{\h})(q^{ t_{\mf h}^{1,2}+\frac12 t_{\mf h}^{2,2}} (1 \ot \sigma_{\h})) &= q^{ (t_{\mf h}^{1,2}+t_{\mf h}^{1,3})+\frac12t_{\mf h}^{2,2}+\frac12t_{\mf h}^{3,3} +  t_{\mf h}^{2,3}} (1 \ot \tsh \ot \tsh) \\
   &= q^{t_{\mf h}^{1,2}+\frac12 t_{\mf h}^{2,2}} (1  \ot \tsh\ot 1) q^{ t_{\mf h}^{1,3}+\frac12 t_{\mf h}^{3,3} } (1 \ot 1 \ot \tsh) q^{ t_{\mf h}^{2,3}}\\
  &= (E_{\h,\sigma})^{1,2}(E_{\h,\sigma})^{1,3}(K^{2,3})^2
 \end{align*}
\end{proof}

\begin{rmk}
 More generally, the above formulaes leads to representations of $B_n^1$ for any automorphism $\sigma_{\h}$ of $U_{\h}(\mf g)$ which satisfies $\sigma_{\h}^{\ot 2}(\R_{\h})=\R_{\h}$, for example:
 \begin{itemize}
  \item Cartan automorphisms, that is of the form $\sigma_{\h}(e_i^{\pm})=\lambda_i^{\pm}e_i^{\pm}$ where $(\lambda_i)_{i=1\dots n}$ is a family of invertible elements of $\C[[\h]]$
  \item diagram automorphisms~\cite{Etingof2000}.
 \end{itemize}

\end{rmk}

\subsection{Cyclotomic KZ connection, analytic representations of $B_n^1$}
Let $\mu_N \subset \C^{\times}$ be the group of $N$th roots of unity. For each $1\leq i,j\leq n$ and each $\zeta \in \mu_N$, define the hyperplane $D_{i,j,\zeta}=\{(z_1,\dots,z_n)\in \C^n\ |\ z_i=\zeta z_j\}\subset \C^{n}$. Let $W_{n,N}$ be the configuration space
\[
 (\C^{\times})^n - \bigcup_{\substack{1\leq i,j\leq n \\ \zeta \in \mu_N}} D_{i,j,\zeta}
\]
The pure braid group associated to $W_{n,N}$ is $P_{n,N}=\pi_1(W_{n,N}, z^* )$ where $z^* =(z^*_1,\dots,z^*_n) \in \RR^n$ is such that $0< z^*_1<\dots <z^*_n$.

Let now $H_{i,j}$ be the hyperplane $\{(z_1,\dots,z_n)\in \C^n\ |\ z_i=z_j\} \subset \C^n$ and $X_n$ be the configuration space
\[
 \Bigl\{ (\C^{\times})^n - \bigcup_{1\leq i,j\leq n} H_{i,j} \Bigr\} \bigl.\bigr/\mf S_n 
\]
The fundamental group $\pi_1(X_n,\mf S_n z^*)$ is $B_n^1 \cong B_{n+1} \times_{\mf S_{n+1}} \mf S_n$, where $\mf S_n$ is the subgroup of the group of permutations of $\{0,\dots,n\}$ which fix 0. The canonical map 
\[
\begin{array}{ccc}
W_{n,N}&\longrightarrow & X_n\\
(z_1,\dots,z_n) &\longmapsto& \left[(z_1^N,\dots,z_n^N)\right]
\end{array}
\]
is the covering corresponding to the group morphism $\phi_{n,N}:B_n^1\rightarrow (\Z/N\Z)^n \rtimes \mf S_n$. We have then $P_{n,N} \cong \ker \phi_{n,N}$.

Let $A_{\kz}=U(\mf g)[[\h]]\rtimes \Z$ where $\ts\cong 1 \in \Z$ acts by $\sigma$. Let $V$ be a finite dimensional $A_{\kz}$-module and $W$ be a finite dimensional $\mf h$-module. If $m$ is the multiplication of $U(\mf h)$, let $t_{\mf h}^{i,i}$ be equal to $m(t_{\mf h}) \in U(\mf h)$ viewed as an element of $End(V)$ acting on the $i$th component of $W \ot V^{\ot n}$ (here $W$ has index 0). In the same way, $t_{\mf h}^{0,i}$ is viewed as an element of $End(W\ot V)$ acting on the 0th and the $i$th component of $W \ot V^{\ot n}$, and $t^{i,j}$ is defined similary. Then, the cyclotomic KZ differential system is
\[
 \frac{\partial H(z_1,\dots, z_n)}{\partial z_i} = \frac{\h}{2\pi\sqrt{-1}} \left( \frac{N(t_{\mf h}^{0,i} + \frac12 t_{\mf h}^{i,i})}{z_i}+ \sum_{j\neq i, j=1}^n \sum_{a \in \Z/N\Z} \frac{(\sigma^a \ot 1)(t^{i,j})}{z_i - \zeta_N^a z_j} \right) H(z_1,\dots,z_n)
\]
where $H$ is a function $H:W_{n,N} \rightarrow W\ot V^{\ot n}[[\h]]$. This system is compatible, and thus defines a flat connection over $W_{n,N}$ with fiber $W\ot V^{\ot n}[[\h]]$. It follows that it induces a monodromy morphism
\[
 \pi_1(W_{n,N}) \longrightarrow GL(W \ot V^{\ot n}[[\h]])
\]
There is a natural action of $G_{n,N}=(\Z/N\Z)^n \rtimes S_n$ on $W_{n,N}$, and the action of $\sigma$ on $V$ induces an action of $G_{n,N}$ on $W \ot V^{\ot n}$ (it acts trivially on $W$).
The KZ connection is $G_{n,N}$-equivariant with respect to these actions, which implies that it also induces a monodromy representation of $B_n^1$. Moreover, as both $t$ and $t_{\mf h}$ are $\mf h$-invariant, so is the KZ system.
\begin{rmk}
 $G_{n,N}$ is a complex reflection group, denoted by $G(N,1,n)$ in the Shephard-Todd classification~\cite{Shephard1954}. in particular, $G_{n,2}$ is the Coxeter group of type $B$. It turns out~\cite{Broue1998} that $D_{i,j,\zeta}$ and $\{z_i=0\}$ are the reflecting hyperplanes of $G_{n,N}$.
\end{rmk}

\subsection{Equivalence of representations}
It follows from the previous sections that for any finite dimensional $ U(\mf g)\rtimes \Z$-module $V$ and any finite dimensional $\mf h$-module $W$, one can construct two representations $\rho_{\h}$ and $\rho_{\kz}$ of $B_n^1$ in $W \ot V^{\ot n}[[\h]]$. The rest of this paper is devoted to a proof of the following theorem:
\begin{thm}
 The representations $\rho_{\h}$ and $\rho_{\kz}$ are equivalent.
\end{thm}

\section{Algebraic construction of representations of $B_n^1$}
The goal of this section is to prove Theorem~\ref{thm:alg-reps} by a direct computation.

 The relation between the $\rho(\sigma_i)$s are satisfied thanks to the QYBE. For $i>1$ the relation 
 \[
 \rho(\tau)\rho(\sigma_i)=\rho(\sigma_i)\rho(\tau)  
 \]
 follows from the QYBE and axiom (c). Thus, it remains to check the relation 
 \[
 \rho(\tau)\rho(\sigma_1)\rho(\tau)\rho(\sigma_1)=\rho(\sigma_1)\rho(\tau)\rho(\sigma_1)\rho(\tau). 
 \]
 Axiom (c) implies that $K^{1,2}K^{1,3}$ commutes with $(R^{2,3})^{-1}$. Therefore, $K^{1,2}K^{1,i}$ comutes with $(R^{2,i})^{-1}$ for $i\in \{3,\dots,n\}$. As $K^{1,j}$ obviously commutes with $(R^{2,i})^{-1}$ if $i\neq j$, it follows that
 \begin{equation}\label{eq:commu}
  \prod_{i=3}^n (R^{2,i})^{-1} \prod_{i=2}^n K^{1,i}= \prod_{i=2}^n K^{1,i} \prod_{i=3}^n (R^{2,i})^{-1} 
 \end{equation}

 Then, $\rho(\sigma_1)\rho(\tau)\rho(\sigma_1)\rho(\tau)$ is equal to:
 \begin{align*}
&(1,2)R^{1,2} \prod_{i=2}^n (R^{1,i})^{-1}\prod_{i=2}^n K^{1,i} E^{0,1}  (1,2)R^{1,2}  \prod_{i=2}^n (R^{1,i})^{-1}\prod_{i=2}^n K^{1,i} E^{0,1}\\
=&(1,2) \prod_{i=3}^n (R^{1,i})^{-1}\prod_{i=2}^n K^{1,i} E^{0,1}  (1,2)  \prod_{i=3}^n (R^{1,i})^{-1}\prod_{i=2}^n K^{1,i} E^{0,1}\\
=& \prod_{i=3}^n (R^{2,i})^{-1}    \prod_{i=3}^n (R^{1,i})^{-1}\prod_{i=2}^n K^{1,i} \prod_{i=3}^n K^{2,i}K^{1,2}E^{0,1}E^{0,2} \\
 \end{align*}
 where the last step follows from~\eqref{eq:commu} and the commutativity of $C$.

On the other hand, $\rho(\tau)\rho(\sigma_1)\rho(\tau)\rho(\sigma_1)$ is equal to
 \begin{align*}
&\prod_{i=2}^n (R^{1,i})^{-1}\prod_{i=2}^n K^{1,i} E^{0,1}  (1,2)R^{1,2}  \prod_{i=2}^n (R^{1,i})^{-1}\prod_{i=2}^n K^{1,i} E^{0,1}(1,2)R^{1,2} \\
=&\prod_{i=2}^n (R^{1,i})^{-1}\prod_{i=2}^n K^{1,i} E^{0,1}   \prod_{i=3}^n (R^{2,i})^{-1}(1,2)\prod_{i=2}^n K^{1,i} E^{0,1}(1,2)R^{1,2} \\
=&\prod_{i=2}^n (R^{1,i})^{-1}\prod_{i=3}^n (R^{2,i})^{-1}\prod_{i=2}^n K^{1,i} E^{0,1}   (1,2)\prod_{i=2}^n K^{1,i} E^{0,1}(1,2)R^{1,2} \\
=&\prod_{i=2}^n (R^{1,i})^{-1}\prod_{i=3}^n (R^{2,i})^{-1}\prod_{i=3}^n K^{1,i}   \prod_{i=3}^n K^{2,i}(K^{1,2})^{2}  E^{0,1}E^{0,2}R^{1,2} \\
 \end{align*}
 Axioms (c) and (d) implies that $K^{1,i}K^{2,i}$ commutes with $R^{1,2}$. Using axiom (b), it follows that:
 \begin{multline*}
 \rho(\tau)\rho(\sigma_1)\rho(\tau)\rho(\sigma_1)=\\
(R^{1,2})^{-1}\prod_{i=3}^n (R^{1,i})^{-1}\prod_{i=3}^n (R^{2,i})^{-1}R^{1,2} \prod_{i=3}^n K^{1,i}   \prod_{i=3}^n K^{2,i}(K^{1,2})^{2}  E^{0,1}E^{0,2}
 \end{multline*}
If $i\neq j$, then $(R^{1,i})^{-1}$ commutes with $(R^{2,j})^{-1}$, implying that
\[
\prod_{i=3}^n (R^{1,i})^{-1}\prod_{i=3}^n (R^{2,i})^{-1}=\prod_{i=3}^n (R^{2,i}R^{1,i})^{-1}
\]
Finally, the QYBE implies that 
\[
(R^{1,2})^{-1}\prod_{i=3}^n (R^{2,i}R^{1,i})^{-1}R^{1,2}= \prod_{i=3}^n (R^{1,i}R^{2,i})^{-1}
\]
which concludes the proof.
\section{Quasi-Reflection Algebras}\label{sec:QRA}
In this section, we recall from~\cite{Enriquez2008} the notion of Quasi-Reflection Algebra (QRA), which is a natural generalization of the notion of Quasi-Triangular Quasi-Bialgebra.

\begin{defi}[\cite{Drinfeld1990}]
Let $A$ be an associative algebra with unit, $\Delta_A$ be an algebra morphism from $A$ to $A^{\ot 2}$, $\R_A \in (A^{\ot 2})^{\times}$ and $\Phi_A \in (A^{\ot 3})^{\times}$. $(A,\Delta_A, \Phi_A,\R_A)$ is a quasi-triangular quasi-bialgebra (QTQBA) if
\begin{align*}
 &\Phi_A^{2,3,4}\Phi_A^{1,23,4}\Phi_A^{1,2,3}=\Phi_A^{1,2,34}\Phi_A^{12,3,4}\\
 &\R_A^{12,3}=\Phi_A^{3,2,1} \R_A^{1,3} (\Phi_A^{1,3,2})^{-1} \R_A^{2,3} \Phi_A^{1,2,3}\\
&\R_A^{1,23}=(\Phi_A^{2,3,1})^{-1} \R_A^{1,3} \Phi_A^{2,1,3} \R_A^{1,2} (\Phi_A^{1,2,3})^{-1}\\
 &\R_A \Delta_A(a) = \Delta_A^{2,1}(a) \R_A,\ \forall a\in A\\
 &(\id_A \ot \Delta_A)\circ \Delta_A(a) =\Phi_A \left[(\Delta_A \ot \id_A)\circ \Delta_A(a)  \right]\Phi_A^{-1},\ \forall a\in A
\end{align*}
\end{defi}

\begin{defi}
A \emph{dynamical pseudo-twist} (DPT) over $A$ is a tuple $(B,\Delta_B,\Psi_B)$ where $B$ is an associative algebra, $\Delta_B$ is an algebra morphism $B \rightarrow B\ot A$ and $\Psi_B$ is an element in $(B \ot A^{\ot 2})^{\times}$ such that:
\begin{equation}
\label{eq:PsiCop}
 (\id_B \ot \Delta_A) \circ \Delta_B(b) = \Psi_B \left[ (\Delta_B \ot \id_A) \circ \Delta_B(b) \right] \Psi_B^{-1},\ \forall b\in B
\end{equation}

\begin{equation}
\label{eq:mixedPenta}
 \Psi_B^{1,2,34} \Psi_B^{12,3,4}=\Phi_A^{2,3,4} \Psi_B^{1,23,4} \Psi_B^{1,2,3}
\end{equation}
\end{defi}

The last relation is called the mixed pentagon relation.
\begin{rmk}
 This is the ``quasi'' version of the notion of comodule-algebra over a bialgebra.
\end{rmk}
\begin{defi}[\cite{Enriquez2008}]
A \emph{Quasi-Reflection Algebra} (QRA) over $A$ is a tuple $(B,\Delta_B,\Psi_B,E_B)$ such that:
\begin{itemize}
 \item $(B,\Delta_B,\Psi_B)$ is a dynamical pseudo-twist over $A$
 \item $E_B$ satisfies the octagon equation 
 \begin{equation}
 \label{eq:octo}
  (\Delta_B \ot \id)(E_B)=\Psi_B^{-1}\R_A^{3,2} \Psi_B^{1,3,2}E_B^{1,3} (\Psi_B^{1,3,2})^{-1}\R_A^{2,3}\Psi_B
 \end{equation}
 and
\begin{equation}\label{eq:Ecop}
\Delta_B(b) E_B = E_B \Delta_B(b),\ \forall b\in B
 \end{equation}
 \end{itemize}
\end{defi}

\begin{prop}
 
The pair $(\R_A,\Phi_A)$ leads to a group morphism $B_n \rightarrow (A^{\ot n}\rtimes S_n)^{\times}$ given by:
\[
 \sigma_i \longmapsto (\Phi_A^{1\dots i-1,i,i+1})^{-1}(i,i+1)\R_A^{i,i+1}\Phi_A^{1\dots i-1,i,i+1} 
\]
 where $\Phi_A^{1\dots i-1,i,i+1}=1$ if $i=1$ and $(\Delta \ot \id_A^{\ot i-2})\dots (\Delta_A \ot \id_A^{\ot 3})\circ(\Delta_A \ot \id_A^{\ot 2})(\Phi_A)$ otherwise. It induces a representation of $B_n$ on $V^{\ot n}$ for each $V \in A\modu$, that is a functor $A\modu\rightarrow B_n\modu$.
 
 In a similar way, a QRA  $(B,\Delta_B,\Psi_B, E_B)$ over $A$ leads to a group morphism $B_n^1 \rightarrow (B \ot A^{\ot n}\rtimes S_n)^{\times}$, the image of the additional Artin generator being defined by:
\[
 \tau \longmapsto (\Psi_B^{0,1,2\dots n}) E_B^{0,1}(\Psi_B^{0,1,2\dots n})^{-1}
\]
where $\Psi_B^{0,1,2\dots n}= (\id_B \ot \id_A \ot \Delta_A \ot \id_A^{\ot n-3})\dots (\id_B \ot \id_A \ot \Delta_A \ot \id_A)\circ (\id_B \ot \id_A \ot \Delta_A)(\Psi_B)$. Therefore, it induces a representation of $B_n^1$ on $W\ot V^{\ot n}$ for each $(W,V) \in B\modu \times A\modu$, that is a functor $B \modu \times A \modu \rightarrow B_n^1\modu$. 
\end{prop}

Recall that if $A=(A,\Delta_A,\Phi_A,\R_A)$ is a QTQBA, the twist of $A$ by an element $F \in (A^{\ot 2})^{\times}$ is the QTQBA $A^F=(A,\Delta_A^F, \Phi_A^F,\R_A^F)$ where
\begin{itemize}
 \item $\Delta_A^F(a)=F^{-1}\Delta_A(a) F$, $a\in A$
\item $\R_A^F=F^{2,1}\R_A F^{-1}$
\item $\Phi_A^F= F^{2,3} F^{1,23} \Phi_A (F^{12,3})^{-1} (F^{1,2})^{-1}$
\end{itemize}

Then, if $B=(B,\Delta_B,\Psi_B,E_B)$ is a QRA over $A$, the twist of $(A,B)$ by $(F,G)$ where $G \in (B\ot A)^{\times}$ is the QRA $B^{(F,G)}=(B,\Delta_B^{G}, \Psi_B^{(F,G)},E_B^{G})$ over $A^F$, where
\begin{itemize}
 \item $\Delta_B^{(F,G)}(b)=G^{-1}\Delta_B(b) G$, $b \in B$
\item $\Psi_B^{(F,G)}=F^{2,3} G^{1,23} \Psi_B (G^{12,3})^{-1} (G^{1,2})^{-1}$
\item $E_B^{G}=G^{-1}E_BG$
\end{itemize}
The categories $A\modu \times B\modu$ and $A^F\modu\times B^{(F,G)}\modu$ identify canonically, because the underlying algebras are the same. The main property of the twist operation is that the following diagram commutes:
\begin{center}
 \begin{tikzpicture}
\matrix (m) [matrix of math nodes, row sep=3em,
column sep=2.5em, text height=1.5ex, text depth=0.25ex]
{ A\modu \times B\modu &   B_n^1\modu \\
A^F\modu \times B^{(F,G)}\modu &  \\ };
\path[->,font=\scriptsize]
(m-1-1) edge node[auto] {} (m-1-2)
(m-2-1) edge[<-] node[above,sloped] {$\simeq$} (m-1-1)
(m-2-1.north east) edge node[auto] {} (m-1-2);
  
 \end{tikzpicture}

\end{center}

 It follows that the corresponding representations of $B_n^1$ are equivalent.
\begin{rmk}
 The defining axioms of QRA have a categorical interpretation. Namely, they imply that $B\modu$ is a braided module category over the braided monoidal category $A\modu$.
\end{rmk}

\section{Algebraic solutions of the ABRR and mixed pentagon equation}
The goal of this section is to give an explicit construction of a QRA over the ``quantum group'' $U_{\h}(\mf g)$. We show that the resulting representations identify with those of Theorem~\ref{thm:rep_gtsig}.

For this, we recall the construction of a dynamical twist based on the quantum Shapovalov form~\cite{Etingof2001}. Whereas the approach of~\cite{Etingof2001} is ``functional'' (i.e. the twist is a map $J:\mf h^* \rightarrow U_q(\mf g)^{\ot 2}$), ours is purely algebraic: the twist is a family $\{ \Psi_{q,m} , m\in \N^r\}$ lying in a localization $U_q(\mf h) \ot U_q(\mf g)^{\ot 2}[1/D^{12}]$ of $U_q(\mf h)\ot U_q(\mf g)^{\ot 2}$. To establish the properties of this twist (namely the ABRR equation (Theorem~\ref{thm:ABRR-Psi}) and the mixed pentagon equation (Theorem~\ref{thm:mixPentaSansShift})), we use an injection
\[
U_q(\mf h) \ot U_q(\mf g)^{\ot 2}[1/D^{12}] \hookrightarrow U'_{\h,loc}(\mf h) \tilde\ot U_{\h}(\mf g)^{\hat\ot 2}[\h^{-1}]
\]
(see Lemma~\ref{lem:QToFormal}).
We then apply to the dynamical twist an algebraic analog of a shift of the functional variable which allows one to avoid the poles of $J(\lambda)$. This leads to a family $\{\Psi_{q,m}^{\sigma_{\nu}}, m\in \N^r\}$ lying in a new localized algebra $U_q(\mf h) \ot U_q(\mf g)^{\ot 2}[1/D_{\nu}^{12}]$. This algebra injects into the \emph{non-localized} algebra $U'_{\h}(\mf h)\hat\ot U_{\h}(\mf g)^{\hat\ot 2}$, which leads to an element $\Psi_{\h,\sigma_{\nu}}$ in this algebra satisfying the mixed pentagon equation and a modified ABRR equation (Theorem~\ref{thm:formalPentaABRR}). The fact that $\Psi_{\h,\sigma_{\nu}}$ is well defined relies on a study of its $\h$-adic valuation (Lemma~\ref{lem:shiftedConvergence}).
This is illustrated in the following diagram:
\begin{center}
 \begin{tikzpicture}
\matrix (m) [matrix of math nodes, row sep=3em,
column sep=8em, text height=1.5ex, text depth=0.25ex]
{\Psi_{q,m} \in U_q(\mf h) \ot U_q(\mf g)^{\ot 2}[1/D^{12}] & U_q(\mf h) \ot U_q(\mf g)^{\ot 2}[1/D_{\nu}^{12}] \ni \Psi_{q,m}^{\sigma_{\nu}}\\
\Psi_{\h} \in U'_{\h,loc}(\mf h) \tilde\ot U_{\h}(\mf g)^{\hat\ot 2}[\h^{-1}]& U'_{\h}(\mf h) \hat\ot U_{\h}(\mf g)^{\hat\ot 2}[\h^{-1}] \ni \Psi_{\h,\sigma_{\nu}}\\ 
};
\path[->,font=\scriptsize]
(m-1-1) edge node[above] {$\simeq$} node[below] {Shift} (m-1-2);
\path[right hook->]
(m-1-1) edge node[left] {} (m-2-1)
(m-1-2) edge node[right] {} (m-2-2)
;
\end{tikzpicture}
\end{center}

\subsubsection*{Conventions}

The Killing form on $\mf g$ is normalized in such a way that $(h_i,h_j)=d_j^{-1}a_{ij}$ and from now on we assume that $t=2(\ ,\ )^{-1}$. It follows that for all $h \in \mf h$, $\alpha_j(h)=(d_jh_j,h)$. Let $\check h_i \in \mf h$ be the unique element such that $(h_i,\check h_j)=\delta_{ij}$, then $t_{\mf h}=2\sum_i \check h_i \ot h_i$. Finally, let $\rho\in \mf h$ be the unique element such that for $j=1\dots r$,  $\alpha_j(\rho)=2d_j$. If $m \in \N^r$, set $\underline{mdh}=\sum_i m_i d_i h_i$ and $\ul{md}=\sum_jm_jd_j$.

\subsection{Quantum groups and localization}
\subsubsection{}The bialgebra $U_{\h}(\mf g)$ admits a ``rational'' form: let $U_q(\mf g)$ be the $\C(q)$-algebra generated by $e_i^{\pm}, k_i^{\pm 1}$ for $i=1\dots r$, with relations
\begin{align*}
 k_ik_j&=k_ik_j & k_i^{\pm 1}k_i^{\mp 1}&=1 \\
 k_j e_i^{\pm} k_j^{-1}&=q^{\pm a_{ji}} e_i & [e_i^+,e_j^-]&=\delta_{ij} \frac{k_i^{d_i} - k_i^{-d_i}}{q^{d_i}-q^{-d_i}}
\end{align*}
and the quantum Serre's relations where $e^{\h}$ is replaced by $q$. The same formulas as for $U_{\h}(\mf g)$ turn $U_q(\mf g)$ into a Hopf algebra. Recall that $U_{\h}(\mf b^{\pm})$ is the subalgebra of $U_{\h}(\mf g)$ generated by $(e_i^{\pm},h_i)_{i=1\dots r}$. $U_q(\mf b^{\pm})$ is defined in a similar fashion. Let $U_q(\mf n^{\pm})$ be the subalgebra of $U_q(\mf g)$ generated by $(e_i^{\pm})_{i=1\dots r}$ and $U_q(\mf h)$ be the subalgebra generated by $(k_i^{\pm 1})_{i=1\dots r}$.

Define a $\Z^r$-grading on $U_q(\mf g)$ by setting $|e_i^{\pm}|=\pm \delta_i$ and $|k_i^{\pm 1}|=0$ where $\delta_i$ is the $i$th basis vector of $\Z^r$. Recall that the element $\bar\R=e^{-\tth/2}\R  \in U_{\h}(\mf g)^{\ot 2}$ can be identified with an invertible element of $U_q(\mf n^+) \hat\ot U_q(\mf n^-)$ where $\hat\ot$ is the completed tensor product over $\C(q)$ with respect to the $\Z^r$-grading.

If $\beta=\sum_{i=1}^{r} n_i\alpha_i$, $n_i \in \Z$, set $k_{\beta}=\prod_{i=1}^r k_i^{n_i}$.

\subsubsection{}In the following, tensor products are understood over $\C(q)$. Recall~\cite[Chap. 9]{Chari1994} that $U_q(\mf g)$ is isomorphic as a $\C(q)$-vector space to
\[
 U_q(\mf h) \ot U_q(\mf n^+) \ot U_q(\mf n^-)
\]
Moreover, it follows from a quantum analog of the PBW Theorem that $U_q(\mf h)$ is isomorphic as an algebra to $\C(q)[k_i^{\pm 1},i=1\dots r]$. $\Z^r$ acts on $U_q(\mf h) \cong \C(q)[k_i^{\pm 1},i=1\dots r]$ by $\delta_j \cdot k_i=q^{a_{i,j}}k_i$. Extend this action to an action of $(\Z^r)^n$ on $U_q(\mf h)^{\ot n+1}$, the first copy of $U_q(\mf h)$ being acted on trivially. Then, the following relation holds for any $x \in  U_{q}(\mf h)^{\ot n+1}$ and any homogeneous element $y \in U_q(\mf h) \ot U_{q}(\mf g)^{\ot n}$:
\begin{equation} \label{eq:comrel}
yx=(|y| \cdot x) y
\end{equation}

Let $A$ be an algebra without zero divisor, recall that a left denominator set for $A$ is a multiplicative subset $D$ of $A-\{0\}$ such that for each $d \in D$, $a\in A$, there exists $a' \in A$, $d' \in D$ satisfying the Ore condition~\cite[Chap. 2]{McConnell2001}: 
\[
a'd=d'a
\]
Then, we have the following generalization of the localization of commutative rings:
\begin{thm}[Ore]
If $D$ is a left denominator set for an algebra $A$ without zero divisor, then there exists a left quotient ring $A[1/D]$ of $A$ in the following sense:
\begin{itemize}
\item there exists an injective algebra morphism $\phi:A \rightarrow A[1/D]$ such that for all $d \in D$, $\phi(d)$ is invertible in $A[1/D]$
\item if $B$ is any algebra and $\theta$ is a morphism $\theta:A \rightarrow B$ such that $\theta(D) \subset B^{\times}$, then there exists a unique morphism $\theta':A[1/D]\rightarrow B$ such that $\theta=\theta' \circ \phi$.
\end{itemize}
\end{thm}
Moreover, it is easily seen that if $\theta$ is injective, then so is $\theta'$. Indeed, if $ad^{-1}$ is an element of the kernel of $\theta'$, then $(ad^{-1})d=a$ belongs to $\ker \theta' \cap A=\ker \theta=\{0\}$, meaning that $a=0$ and therefore that $\ker \theta'=\{0\}$.   In order to apply the Ore's Theorem to $U_q(\mf g)$, we need the following:
\begin{lem}\label{lem:OreCond}
Let $A=U_q(\mf h)\ot \bigotimes_{i=1}^n U_i$ where $U_i$ is either $U_q(\mf g)$ or  $U_q(\mf b^{\pm})$, and $D$ be a multiplicative subset of $U_q(\mf h)^{\ot n+1}-\{0\}$ which is stable under the action of $(\Z^r)^n$ on it last $n$ components. Then, $D$ is a left denominator set for $A$.
\end{lem}
\begin{proof}
If $x$ is an homogeneous element of $A$, the Ore condition follows from~\eqref{eq:comrel} and from the fact that $|x|\cdot d \in D$ because $D$ is stable under the action of $(\Z^r)^n$. If $x$ is the sum of two homogeneous elements $x_1,x_2$ and $d \in D$, let
\[
x'=(|x_2|\cdot d)x_1+(|x_1|\cdot d)x_2
\] 
As for $i=1,2$, $|(|x_i|\cdot d)|=0$ it follows that:
\begin{align*}
x'd&=((|x_1|+|x_2|)\cdot d)x_1+((|x_1|+|x_2|)\cdot d)x_2\\
&=((|x_1|+|x_2|)\cdot d)x
\end{align*}
The same argument can obviously be applied in the general case of an arbitrary sum of $n$ homogeneous elements.
\end{proof}
For $\beta \in R^+,i \in \N$ let 
\[
D_{\beta,i}=1-k_{\beta}q^{-\frac i2(\beta,\beta)} \in U_q(\mf h)
\]
and let
\[
D^{12}=\langle (\{0\} \times \Z^r) \cdot \{\Delta_q(D_{\beta,i}), \beta\in R^+, i\in \N\}\rangle \subset U_q(\mf h)^{\ot 2}
\]
and
\[
D^{123}=\langle (\{0\} \times (\Z^r)^2) \cdot \{(\id \ot \Delta_q)\circ\Delta_q(D_{\beta,i}), \beta \in R^+,i\in \N\}\rangle \subset U_q(\mf h)^{\ot 3}
\]
where $\langle X\rangle$ denotes the multiplicative subset spanned by $X$.
\begin{lem}
There are algebra morphisms
\[
U_q(\mf h)\ot U_q(\mf b^+)\ot U_q(\mf n^-)[1/D^{12}] \rightarrow U_q(\mf h)\ot U_q(\mf b^+)\ot U_q(\mf g)\ot  U_q(\mf n^-)[1/D^{12}, 1/D^{123}]
\]
induced by 
\begin{align*}
 X &\mapsto X^{0,1,2} &  X &\mapsto X^{0,12,3} \\
 X &\mapsto X^{0,1,23} &  X &\mapsto X^{01,2,3} \\
\end{align*}
\end{lem}
\begin{proof}
The image of $D^{12}$ by each of these maps is a subset either of $D^{12}$ or of $D^{123}$, which contains only invertible elements by construction. Thus, the proof follows from the universal property of the localization.
\end{proof}
\subsubsection{}Let $U_{\h}'(\mf h)$ be the Quantum Formal Series Hopf Algebra~\cite{Gavarini2002} associated to $U_{\h}(\mf h)$, that is the subalgebra of $U_{\h}(\mf h)$ generated by $\h h_i,\ i=1\dots r$. It is a flat deformation of the algebra $\widehat{S(\mf h)}$ of formal functions on $\mf h$. Indeed, there is an algebra isomorphism
\[
\begin{array}{rcl}
U_{\h}'(\mf h)& \longrightarrow & \C[[u_1,\dots,u_r,\h]]\\
\h h_i & \longmapsto & u_i
\end{array}
\]
For $\beta=\sum_{i=1}^r \beta_i \alpha_i \in R^+$, let
\[
\ell_{\beta}=\sum_{i=1}^r \beta_i d_i u_i 
\]
and set
\[
U'_{\h,loc}(\mf h)=\C[[u_1,\dots,u_r]][\frac1{\ell_{\beta}},\beta \in R^+][[\h]]
\]
Let $A$ be a topologically free $\C[[\h]]$-algebra. Recall that it means that there exists a $\C$-vector space $A_0$ such that $A$ is isomorphic to $A_0[[\h]]$ as a $\C[[\h]]$-module. We define the completed tensor product
\[
U'_{\h,loc}(\mf h) \tilde \ot A = A_0[[u_1,\dots,u_r]] [\frac1\ell_{\beta},\beta \in R^+][[\h]]
\]

\begin{lem}
The coproduct $\Delta_{\h}:U_{\h}(\mf h)\rightarrow U_{\h}(\mf h)^{\hat\ot 2}$ induces an algebra morphism
\[
U_{\h,loc}'(\mf h) \longrightarrow U_{\h,loc}'(\mf h) \tilde\ot U_{\h}(\mf h)
\]
\end{lem}
\begin{proof}
Under the identification $U_{\h}'(\mf h)\cong \C[[u_1,\dots,u_n,\h]]$, the coproduct of $U_{\h}(\mf h)$ becomes an algebra morphism $\C[[u_1,\dots,u_n,\h]]\rightarrow\C[[u_1,\dots,u_n,\h]]\hat \ot U_{\h}(\mf h)$ induced by
\[
u_i \longmapsto u_i \ot 1 + 1 \ot \h h_i
\]
which induces an algebra morphism $\C[[u_1,\dots,u_n,\h]]\rightarrow\C[[u_1,\dots,u_n]][\frac1{\ell_{\beta}}][[\h]]\tilde \ot U_{\h}(\mf h)$. As the image of $\ell_{\beta}$ ($\beta=\sum_i \beta_i \alpha_i$) through this map is invertible, its inverse being
\[
\sum_{k\geq 0} \h^k (-1)^k \frac1{\ell_{\beta}^{k+1}} \ot (\sum_i (\beta_id_ih_i))^k,
\]
it extends to an algebra morphism
\[
U_{\h,loc}'(\mf h) \longrightarrow U_{\h,loc}'(\mf h) \tilde\ot U_{\h}(\mf h)
\]
\end{proof}
\subsubsection{} There is a field morphism $\C(q)\rightarrow \C[[\h]][\h^{-1}]$ given by $q \mapsto e^{\h}$, and a compatible injective Hopf algebra morphism
\[
U_q(\mf g) \rightarrow U_{\h}(\mf g)[\h^{-1}]
\]
given by
\begin{align*}
e_i^{\pm} &\mapsto e_i^{\pm}&  k_i & \mapsto e^{\h d_i h_i}
\end{align*}
\begin{lem}\label{lem:QToFormal}
There are injective algebra morphisms
\[
U_q(\mf h)\ot U_q(\mf b^+) \ot U_q(\mf n^-)[1/D^{12}] \rightarrow U'_{\h,loc}(\mf h) \tilde\ot (U_{\h}(\mf b^+) \hat\ot U_{\h}(\mf n^-))[\h^{-1}]
\]
and
\[
U_q(\mf h)\ot U_q(\mf b^+)\ot U_q(\mf g) \ot U_q(\mf n^-)[1/D^{12},1/D^{123}] \rightarrow U'_{\h,loc}(\mf h) \tilde\ot (U_{\h}(\mf b^+)\hat\ot U_{\h}(\mf g) \hat\ot U_{\h}(\mf n^-))[\h^{-1}]
\]
which commute with the insertion-coproduct morphisms.
\end{lem}
\begin{proof}
The image of $D^{12}$ in $U'_{\h}(\mf h) \hat \ot U_{\h}(\mf b^+)\hat \ot U_{\h}(\mf n^-)[\h^{-1}]$ is the multiplicative part spanned by
\begin{multline}\label{eq:invSet}
\left\{1-\left(\exp(\ell_{\beta}) \ot \exp(\h\sum_i\beta_id_ih_i)\ot 1\right)\right.\\ \left.\times\exp\left(-\h\left(\frac i2 (\beta,\beta)+\sum_{j,k}\beta_jm_k a_{jk}\right)\right)\ \vert \ \beta \in R^+,i\in \N \right\}
\end{multline}
These elements actually belong to $U'_{\h}(\mf h) \hat \ot U_{\h}(\mf b^+)\hat \ot U_{\h}(\mf n^-)$ and are invertible in this algebra, as they reduce to elements of the form $1-\exp(\ell_{\beta})$ modulo $\h$, the latter being invertible in $\C[[u_1,\dots,u_r]][\frac1{\ell_{\beta}}]$ by construction. Hence elements of~\eqref{eq:invSet} are invertible, which implies that the above morphism factors through $U_q(\mf h) \ot U_q(\mf b^+)\ot U_q(\mf n^-)[1/D^{12}]$. The same argument apply to the second case, and the compatibility with the insertion-coproduct maps is obvious.
\end{proof}
\subsubsection{Shift of a localization}\label{subsubsec:ShiftLoc}For $\nu=(\nu_1,\dots,\nu_r) \in (\C^{\times})^r$ define the automorphism $\Sh_{\nu}$ of $U_q(\mf h)$ by 
\[
\Sh_{\nu}(k_i^{\pm 1})=\nu_i^{\pm 1}k_i^{\pm 1}
\]
and set $D_{\nu}^{12}=(\Sh_{\nu}\ot \id) (D^{12})$. This is again a Ore set for $U_q(\mf h) \ot U_q(\mf b^+) \ot U_q(\mf n^-)$ thanks to Lemma~\ref{lem:OreCond}. By construction, $\Sh_{\nu}\ot \id^{\ot 2}$ is an algebra isomorphism
\[
U_q(\mf h)\ot U_q(\mf b^+) \ot U_q(\mf n^-)[1/D^{12}] \rightarrow U_q(\mf h)\ot U_q(\mf b^+) \ot U_q(\mf n^-)[1/D_{\nu}^{12}]
\]
The following lemma shows that by applying an appropriate shift, the above localized "rational" algebra can be injected into the \emph{non-localized} "formal" algebra:
\begin{lem}\label{lem:morphShift}
Assume that $\nu_{\beta}\neq 1$ for any $\beta \in R^+$, then the natural map 
\[
U_q(\mf h) \ot U_q(\mf b^+) \ot U_q(\mf n^-) \rightarrow U'_{\h}(\mf h)\hat\ot U_{\h}(\mf b^+) \hat \ot U_{\h}(\mf n^-)[\h^{-1}]
\]
extends to an injective algebra morphism
\[
U_q(\mf h) \ot U_q(\mf b^+) \ot U_q(\mf n^-)[1/D^{12}_{\nu}]\rightarrow U'_{\h}(\mf h) \hat\ot U_{\h}(\mf b^+) \hat\ot U_{\h}(\mf n^-)[\h^{-1}]
\]
\end{lem}
\begin{proof}
The set $D_{\nu}^{12}$ is multiplicatively generated by the
\[
\left\{ 1-\nu_{\beta}q^{-\frac i2 (\beta,\beta)} (k_{\beta}\ot k_{\beta}) q^{\langle z,\beta\rangle}, \beta \in R^+,i \in \N, z\in \Z^r\right\}
\]
whose image in $U'_{\h}(\mf h)\ot U_{\h}(\mf b^+) \ot U_{\h}(\mf n^-)[\h^{-1}]$ is
\[
\left\{ 1-\nu_{\beta}q^{-\frac i2 (\beta,\beta)} (e^{u_{\beta}}\ot e^{\h d_{\beta}h_{\beta}}) q^{\langle z,\beta\rangle}, \beta \in R^+,i \in \N, z\in \Z^r\right\}
\]
The reduction modulo $\h$ of each element of the latter set is of the form
\[
1-\nu_{\beta} e^{u_{\beta}} \ot 1\ot 1
\]
which is invertible in $\C[[u_1,\dots, u_r]] \tilde\ot U(\mf b^+) \ot U(\mf n^-)$ because $1-\nu_{\beta} \neq 0$. Hence the Lemma follows from the universal property of localization.

\end{proof}
\subsection{Construction of the dynamical twist}
Following~\cite{Enriquez2005,Enriquez2007a}, we will construct a dynamical twist using a quantum analog of the Shapovalov form, prove that it satisfies the ABRR equation~\cite{Arnaudon1998} and give a direct proof that it satisfies the mixed pentagon equation.

For $m\in \Z^r$, set $\mu(m)=\sum_{i,j} m_im_j d_i a_{ij}$. 
While $q^{\tth}$ (resp. $q^{\tth^{1,1}}$) does not belong to $U_q(\mf g)^{\ot 2}$ (resp. $U_q(\mf g)$) or to one of its completion, conjugation by these elements still make sense in the "rational" setup. Namely, there are automorphisms $\ul{\Ad(q^{\tth})}$ of $U_q(\mf g)^{\ot 2}$ and $\ul{\Ad(q^{\tth^{1,1}/2})}$ of $U_q(\mf g)$ uniquely defined by
\begin{align*}
\ul{\Ad(q^{\tth})}(e_j^{\pm} \ot 1)&= e_j^{\pm} \ot k_j^{\pm 2} & \ul{\Ad(q^{\tth})}(1 \ot e_j^{\pm})&=k_j^{\pm 2}\ot e_j^{\pm} \\
\ul{\Ad(q^{\tth})}(X)&=X & \forall X \in U_q(\mf h)^{\ot 2}, j\in \{1,\dots, r\}
\end{align*}
and for any $x_m \in U_q(\mf g)[m],m\in \Z^r$
\begin{align*}
\ul{ \Ad(q^{\tth^{1,1}/2})}(x_m) &=x_m  k^{ 2 \ul m}q^{\mu(m)}
\end{align*}
where $k^{\ul m}=\prod_j k_j^{m_j}$.
\subsubsection{Quantum Shapovalov form}
Let $H$ be the unique linear map $U_q(\mf g)\rightarrow U_q(\mf h)$ such that for all $x_{\pm} \in U_q(\mf n^{\pm})$, $x_0 \in U_q(\mf h)$
\[
H(x_-x_0x_+)=\epsilon(x_-)\epsilon(x_+)x_0
\]
It follows from the definition that $H$ is $U_q(\mf h)$-linear (on the left and on the right). Moreover, if $x_m$, $y_{n}$ are homogeneous elements of $U_q(\mf n^+)$ and $U_q(\mf n^-)$ of degrees $m$ and $n$ respectively, then $H(x_my_n) =0$ if $m \neq -n$. Thus, it induces a family of pairings
\[
\begin{array}{ccl}
U_q(\mf b^+)[m] \times U_q(\mf b^-)[-m] & \longrightarrow & U_q(\mf h) \\
(x_m,y_{-m}) & \longmapsto & H(x_my_{-m})
\end{array}
\]
which restrict to a family of pairings $U_q(\mf n^+)[m] \times U_q(\mf n^-)[-m] \rightarrow U_q(\mf h)$.

Let $K_m$ be the inverse element of this form, that is\footnote{If $V_1,\dots,V_k$ are vector spaces or modules, we abbreviate a general element $a=\sum_{i\in I} a_1(i)\ot \dots\ot a_k(i)$ of $V_1\ot\dots\ot V_k$ by $a=a^{[1]}\ot \dots \ot a^{[k]}$}: $ K_m=K_m^{[1]} \ot K_m^{[2]} \ot K_m^{[3]}$ is the unique element of $ U_q(\mf n^-)[-m]\ot U_q(\mf n^+)[m]\ot \C(q,k_i^{\pm})$ such that for all $x_{\pm} \in U_q(\mf b^{\pm})[\pm m]$,
\begin{equation}
H(x_+x_-)=  H(x_+K^{[1]}_m)K_m^{[3]}H(K_m^{[2]}x_-)
\end{equation}
Set 
\[
\bar \R=q^{-\tth/2}\R=\sum_{m\in \N^r} \bar \R_m=\sum_{m\in \N^r} \bar \R^{[1]}_{m} \ot \bar \R^{[2]}_{m} \in U_q(\mf n^+) \ot U_q(\mf n^-)
\]
and
\[
\bar \R^{-1}=\sum_{m\in \N^r} (\bar \R^{-1})_m=\sum_{m\in \N^r} r^{[1]}_{m} \ot r^{[2]}_{m}.
\]
The following result of De Concini and Kac will be crucial for this section:
\begin{thm}[\cite{De1990}]
Up to some invertible element of $\C(q)$ depending on the choice of basis of  $U_q(\mf n^+)[m]$ and $U_q(\mf n^-)[-m]$, the determinant of $H$ restricted to $U_q(\mf n^+)[m] \times U_q(\mf n^-)[-m]$ (viewed as a matrix with coefficient in $\C(q,k_1,\dots,k_r)$) is
\[
\det_m = \prod_{\beta \in R^+} \prod_{i\in \N} (1-k_{\beta}q^{-(\beta,\rho)-\frac i2(\beta,\beta)})^{P(\sum_jm_j\alpha_j-i\beta)}
\]
where $P(\beta)$ is the Kostant partition function, that is the number of ways one can write $\beta$ as a integral linear combination of the positive roots with non-negative coefficients.
\end{thm}
As a consequence, $K_m$ actually belongs to
 \[
U_q(\mf n^-)\ot U_q(\mf n^+)\ot U_q(\mf h)[(1-k_{\beta}q^{-(\beta,\rho)-\frac i2(\beta,\beta)})^{-1},\beta \in R^+,i\in\N].
\]
\begin{prop}\label{prop:ABRR-K}
The elements $\{K_m, m\in \N^r\}$ satisfy the following relation:
\begin{multline}\label{eq:ABRR-K}
K_{m} =\sum_{m'+m''=m} (\bar \R^{[2]}_{m'} \ot 1\ot 1) (S^2\ot \id\ot \id)( K_{m''}) (1\ot \bar \R^{[1]}_{m'} \ot 1)\\ (1\ot 1\ot k^{-2\ul{m''}} q^{\mu(m'')} k^{-\ul{m'}}  q^{\sum_{i,j} d_im'_im_j''a_{ij}})
\end{multline}
\end{prop}
\begin{proof}
In the following, we identify $K_m$ with its image through the injective algebra morphism
\begin{multline*}
U_q(\mf n^-)\ot U_q(\mf n^+)\ot U_q(\mf h)[(1-k_{\beta}q^{-(\beta,\rho)-\frac i2(\beta,\beta)})^{-1},\beta \in R^+,i\in\N] \\
 \rightarrow U_{\h}(\mf n^-) \hat\ot U_{\h}(\mf n^+) \tilde\ot U'_{\h,loc}(\mf h)[\h^{-1}]
\end{multline*}
the existence of which follows as in Lemma~\ref{lem:QToFormal} from the fact that $1-k_{\beta}q^{-(\beta,\rho)-\frac i2(\beta,\beta)}\mapsto 1-e^{\ell_{\beta}}+O(\h)$.

By construction there are commuting squares
\begin{center}
 \begin{tikzpicture}
\matrix (m) [matrix of math nodes, row sep=3em,
column sep=8em, text height=1.5ex, text depth=0.25ex]
{ U_q(\mf g)& U_q(\mf g)  \\
 U_{\h}(\mf g)[\h^{-1}]&   U_{\h}(\mf g)[\h^{-1}]\\
};
\path[->,font=\scriptsize]
(m-1-1) edge node[auto] {$\ul{\Ad(q^{\tth^{1,1}})}$} (m-1-2)
(m-2-1) edge node[below] {$\Ad(e^{\h\tth^{1,1}})$} (m-2-2)
(m-1-1) edge node[left] {} (m-2-1)
(m-1-2) edge node[right] {} (m-2-2)
;
\end{tikzpicture}
\end{center}
and
\begin{center}
 \begin{tikzpicture}
\matrix (m) [matrix of math nodes, row sep=3em,
column sep=8em, text height=1.5ex, text depth=0.25ex]
{ U_q(\mf g)^{\ot 2}& U_q(\mf g)^{\ot 2}  \\
 U_{\h}(\mf g)^{\hat\ot 2}[\h^{-1}]&   U_{\h}(\mf g)^{\hat\ot 2}[\h^{-1}]\\
};
\path[->,font=\scriptsize]
(m-1-1) edge node[auto] {$\ul{\Ad(q^{\tth})}$} (m-1-2)
(m-2-1) edge node[below] {$\Ad(e^{\h\tth})$} (m-2-2)
(m-1-1) edge node[left] {} (m-2-1)
(m-1-2) edge node[right] {} (m-2-2)
;
\end{tikzpicture}
\end{center}

Set $u=\sum_{m} r^{[1]}_{m}q^{\tth^{1,1}/2}S^{-1}(r^{[2]}_{m})$, according to~\cite{Drinfeld1989b} $q^{\rho}S^{-1}(u)$ is a central element of $U_{\h}(\mf g)$. Let $x_{\pm} \in U_{\h}(\mf n^{\pm})$, we compute $H(x_+S^{-1}(u)x_-)$ in two different ways.
On the one hand, 
\begin{align*}
 H(x_+S^{-1}(u)x_-)&=H(x_+q^{-\rho}x_-q^{\rho}S^{-1}(u))\\
&=H(x_+q^{-\rho}x_-q^{\rho}q^{\tth^{1,1}/2})
\end{align*}
because $S^{-1}(u)$ is of the form $q^{\tth^{1,1}/2}+U_{\h}(\mf b^-)[<0] U_{\h}(\mf b^+)[>0]$. Thus,
\begin{align*}
H(x_+S^{-1}(u)x_-)&=\sum_{m} H(x_+q^{-\rho}K_{m}^{[1]})  K^{[3]}_{m} H(K^{[2]}_{m}x_-)q^{\rho}q^{\tth^{1,1}/2}\\
&=\sum_{m} H(x_+q^{-\rho}K_{m}^{[1]}q^{\rho})  K^{[3]}_{m} q^{\tth^{1,1}/2}H(K^{[2]}_{m}x_-)
\end{align*}
On the other hand, 
\begin{align*}
  H(x_+S^{-1}(u)x_-)&=\sum_{m'} H(x_+ S^{-2}(r^{[2]}_{m'})q^{\tth^{1,1}/2}S^{-1}(r^{[1]}_{m'})x_-)\\
&=\sum_{m',m''} H(x_+ S^{-2}(r^{[2]}_{m'})q^{\tth^{1,1}/2}K^{[1]}_{m''})K^{[3]}_{m''}H(K^{[2]}_{m''}S^{-1}(r^{[1]}_{m'})x_-)
\end{align*}
Recall that $(S^{-1} \ot \id)(\R^{-1})=\R$, thus
\begin{align*}
 (S^{-1} \ot \id)(\bar \R^{-1}q^{-\tth/2}) &= q^{\tth/2}\bar \R \\
\sum_{m} (1 \ot r^{[2]}_{m})q^{\tth/2} (S^{-1}(r^{[1]}_{m})\ot 1) &= q^{\tth/2} \sum_{m} \bar \R^{[1]}_{m} \ot \bar \R^{[2]}_{m}\\
S^{-1}(r^{[1]}_{m}) \ot r^{[2]}_{m} &= q^{-\ul{dmh}} \bar \R^{[1]}_{m} \ot \bar \R^{[2]}_{m}
\end{align*}
Under the normalization chosen at the beginning of this section, the following relation hold:
\begin{align*}
[t_{\mf h}/2, e_j^{\pm}\ot 1]&= \pm e_j^{\pm} \ot \sum_i \alpha_j(\check h_i)h_i\\
&= \pm e_j^{\pm} \ot d_jh_j
\end{align*}
\begin{lem}
Let $x_{m}^{\pm}$ be an homogeneous element of degree $\pm m$ in $U_{\h}(\mf n^{\pm})$, then
\begin{align*}
[\tth^{1,1}/2,x_{m}]&=x_{m} \sum_i m_id_i(\pm 2h_i + \sum_j m_j a_{ij} )\\
&=x_m (\pm 2 \ul{mdh}+\mu(m))
\end{align*}
\end{lem}
\begin{proof}
It follows from a direct computation:
\begin{align*}
\sum_i h_i\check h_i x_{m}^{\pm} &= \sum_i h_ix_{m}^{\pm}(\check h_i \pm \sum_j m_j\alpha_j(\check h_i))\\
&= \sum_i h_ix_{m}^{\pm}(\check h_i \pm m_id_i))\\
&= \sum_i x_{m}^{\pm}(h_i \pm \sum_j m_j\alpha_j(h_i))(\check h_i \pm m_id_i)\\
&= \sum_i x_{m}^{\pm}h_i\check h_i + \sum_i \pm (m_i d_i h_i + \check h_i\sum_j m_j\alpha_j(h_i)) + \sum_i (m_id_i\sum_j m_ja_{ij})\\
&= \sum_i x_{m}^{\pm}h_i\check h_i + \sum_i m_id_i (\pm 2 h_i  + \sum_j m_ja_{ij})
\end{align*}
\end{proof}
Together with the fact that for any element $x \in U_{\h}(\mf g)$, $S^{-2}(x)=\Ad(q^{-\rho})(x)$, it implies that
\begin{flalign*}
H(x_+S^{-1}(u)x_-)&=\sum_{m',m''} H(x_+ q^{-\rho}\bar \R^{[2]}_{m'}q^{\rho}q^{\tth^{1,1}/2}K^{[1]}_{m''})K^{[3]}_{m''}H(K^{[2]}_{m''}q^{-\ul{dm'h}}\bar \R^{[1]}_{m'}x_-)\\
&=\sum_{m',m''} H(x_+ q^{-\rho}\bar \R^{[2]}_{m'}q^{\rho}K^{[1]}_{m''})K^{[3]}_{m''}q^{\tth^{1,1}/2}q^{-2\ul{dm''h}} q^{\mu(m'')} \\
&\hspace{10em} \times  q^{-\ul{dm'h}}  q^{\sum_{i,j} d_im'_im_j''a_{ij}}H(K^{[2]}_{m''}\bar \R^{[1]}_{m'}x_-)
\end{flalign*}
By the uniqueness of $K_{m}$, it follows that
\begin{multline*}
K_{m}^{[1]} \ot K^{[2]}_{m} \ot K^{[3]}_{m} =\sum_{m'+m''=m} \bar \R^{[2]}_{m'} q^{\rho} K^{[1]}_{m''} q^{-\rho}  \ot K^{[2]}_{m''} \bar \R^{[1]}_{m'} \\ \ot K^{[3]}_{m''}q^{-2\ul{dm''h}} q^{\mu(m'')} q^{-\ul{dm'h}}  q^{\sum_{i,j} d_im'_im_j''a_{ij}}
\end{multline*}
\end{proof}
\subsubsection{The dynamical twist}

Let $\rho^*\in \mf h^*$ be the dual element of $\rho$ with respect to the Killing form. Then, 
\begin{equation}\label{eq:rho-t}
(\frac12\rho^* \ot \id)(\tth)=\rho.
\end{equation}
Define an algebra automorphism $\sh_{\rho^*}:U_q(\mf h) \rightarrow U_q(\mf h)$ by:
\[
\sh_{\rho^*}(k_i^{\pm 1})=q^{\pm\frac12 \rho^*(d_ih_i)}k_i^{\pm 1}
\]
For any $m \in \N^r$, let
\[
J_m= S(K^{[3]^{(1)}}_{m}) \ot  S(K^{[2]}_{m})S(K_{m}^{[3]^{(2)}})\ot K^{[1]}_{m}
\]
The map $\sh_{\rho^*}$ induces an isomorphism
\begin{multline*}
\sh_{\rho^*} \ot \id^{\ot 2}: U_q(\mf h) \ot U_q(\mf b^+) \ot U_q(\mf n^-)[1/(\sh_{\rho}^{-1}\ot \id)(D^{12} )] \\ \rightarrow U_q(\mf h) \ot U_q(\mf b^+) \ot U_q(\mf n^-)[1/D^{12}]
\end{multline*}

Thus define\footnote{According to~\cite{Enriquez2005,Enriquez2007a}, the map $a\ot b\ot c\mapsto S(c^{(1)}) \ot S(b)S(c^{(2)})\ot a$ applied to $K$ already leads to a dynamical twist. The two other transformations are needed here because we want a specific form of the ABRR equation.}
\[
\Psi_{q,m}=(\sh_{\rho^*}\ot \id^{\ot 2})\circ (\ul{\Ad(q^{-\tth^{1,2}/2})}\ot \id)(J_m) 
\]
\begin{thm}\label{thm:ABRR-Psi}
The elements $\{\Psi_{q,m}, m\in \N^r\}$ are $\mf h$-invariant elements of $U_q(\mf h)\ot U_q(\mf b^+)\ot U_q(\mf n^-)[1/D^{12}]$ and satisfy the ABRR equation
\begin{equation}
\Psi_{q,m}=\sum_{m'+m''=m}(\bar \R^{-1})_{m'}^{2,3} \ul{\Ad(q^{\tth^{1,2} + \tth^{2,2}/2})} (\Psi_{q,m''})
\end{equation}
\end{thm}
\begin{proof}
The first statement is true because 
\[
(\sh_{\rho^*} \ot \id)\circ \ul{\Ad(q^{-\tth^{1,2}/2})} (\det_m^{12}) \in D^{12}
\]
by construction.
Thus, $\Psi_{q,m}$ can be identified with an element of $U'_{\h,loc}(\mf h) \tilde \ot U_{\h}(\mf b^+)\hat\ot U_{\h}(\mf n^-)[\h^{-1}]$ thanks to Lemma~\ref{lem:QToFormal}, and we will prove the above equation in the latter algebra.
Set $J_m= J^{[1]}_{m} \ot J^{[2]}_{m} \ot J^{[3]}_{m}$. Applying the transformation from $K_m$ to $J_m$ to~\eqref{eq:ABRR-K}, one has:
\begin{multline*}
 J^{[1]}_{m} \ot J^{[2]}_{m} \ot J^{[3]}_{m}=\sum_{m'+m''=m} J^{[1]}_{m} q^{2\ul{dm''h}} q^{\ul{dm'h}}q^{\mu(m'')} q^{\sum_{i,j} d_im'_im_j''a_{ij}}\\ \ot S(\bar \R^{[1]}_{m'})J^{[2]}_{m''} q^{2\ul{dm''h}} q^{\ul{dm'h}}\ot  \bar \R^{[2]}_{m'} q^{\rho}J^{[3]}_{m''} q^{-\rho}
\end{multline*}
that is
\[
 J^{[1]}_{m} \ot J^{[2]}_{m} \ot J^{[3]}_{m}=\sum_{m'+m''=m} J^{[1]}_{m} q^{2\ul{dm''h}} q^{\ul{dm'h}}q^{\mu(m'')} \ot S(\bar \R^{[1]}_{m'})q^{\ul{dm'h}} J^{[2]}_{m''} q^{2\ul{dm''h}} \ot  \bar \R^{[2]}_{m'} q^{\rho}J^{[3]}_{m''} q^{-\rho}
\]
As
\[
\sum_{m} q^{-\ul{dmh}} \bar \R^{[1]}_{m} \ot \bar \R^{[2]}_{m} = \sum_{m} S^{-1}(r^{[1]}_{m}) \ot r^{[2]}_{m}
\]
it follows that
\[
\sum_{m} S(\bar \R^{[1]}_{m}) q^{\ul{dmh}} \ot \bar \R^{[2]}_{m} = \sum_{m} r^{[1]}_{m} \ot r^{[2]}_{m}
\]
Hence,
\begin{align*}
 J^{[1]}_{m} \ot J^{[2]}_{m} \ot J^{[3]}_{m}&=\sum_{m'+m''=m} J^{[1]}_{m} q^{2\ul{dm''h}} q^{\ul{dm'h}}q^{\mu(m'')} \ot r^{[1]}_{m} J^{[2]}_{m''} q^{2\ul{dm''h}} \ot  r^{[2]}_{m'} q^{\rho}J^{[3]}_{m''} q^{-\rho}\\
&=\sum_{m'+m''=m} q^{\tth^{1,2}/2} (1\ot r^{[1]}_{m'}\ot r^{[2]}_{m'})   q^{-\tth^{1,2}/2} q^{\tth^{1,2}+\tth^{2,2}/2}\\ & \hspace{9em} \times (J^{[1]}_{m''} \ot J^{[2]}_{m''} \ot q^{\rho}J^{[3]}_{m''}q^{-\rho}) q^{-\tth^{1,2}-\tth^{2,2}/2} 
\end{align*}
Multiplying the two sides of the last line by $q^{\tth^{1,2}/2}$ on the right and using the $\mf h$-invariance of $J$, the following relation holds
\[
q^{-\tth^{1,2}/2} J_mq^{\tth^{1,2}/2} =\sum_{m=m'+m''}(\bar \R_{m'}^{-1})^{2,3}   q^{\tth^{1,2}+\tth^{2,2}/2-\rho^{(2)}} (q^{-\tth^{1,2}/2}J_{m''} q^{\tth^{1,2}/2} ) q^{-\tth^{1,2}-\tth^{2,2}/2+\rho^{(2)}} 
\]
that is, thanks to relation~\eqref{eq:rho-t},
\[
\Psi_{q,m}=\sum_{m=m'+m''}(\bar \R^{-1})_{m'}^{2,3}   q^{\tth^{1,2}+\tth^{2,2}/2}\Psi_{q,m''}  q^{-\tth^{1,2}-\tth^{2,2}/2} 
\]
as wanted.
\end{proof}

\begin{cor}\label{cor:convergence}
Let $\Psi_{\h,m}$ be the image of $\Psi_{q,m}$ in $U_{\h,loc}(\mf h)\tilde \ot U_{\h}(\mf b^+) \hat \ot U'_{\h}(\mf n^-)[\h^{-1}]$, the sum
\[
\Psi_{\h}=\sum_{m \in \N^r}\Psi_{\h,m}
\]
is convergent in the $\h$-adic topology and has nonnegative $\h$-adic valuation. In other words, $\Psi_{\h}\in U'_{\h,loc}(\mf h)\tilde \ot U_{\h}(\mf b^+) \hat \ot U_{\h}(\mf n^-)$.
\end{cor}
\begin{proof}
Set 
\[
A_{\h}= U'_{\h,loc}(\mf h)\tilde \ot U_{\h}(\mf b^+) \hat \ot U_{\h}(\mf n^-)
\]
We will show that for any $m \in \N^r$, $\Psi_{\h,m}\in \h^{|m|} A_{\h}$. This will be proved by induction on $\omega=|m|$. If $\omega=0$, this is true because $\Psi_{\h,0}=1$. Assume that this holds for any $m$ such that $|m|<\omega$. Let $m_0$ be such that $|m_0|=\omega$, according to the ABRR equation
\[
(\id-\Ad(q^{\tth^{1,2}+\frac12 \tth^{2,2}}))(\Psi_{\h,m_0})=\sum_{\substack{m'+m''=m_0\\(m',m'')\neq (0,0)}} (\bar\R^{-1})_{m'}^{2,3} \Ad(q^{\tth^{1,2}+\frac12 \tth^{2,2}})(\Psi_{\h,m''})
\]
that is
\begin{multline}\label{eq:hbarABRR}
\Psi_{\h,m_0}(1-e^{2\ul{dm_0u}}\ot q^{\mu(m_0)}q^{2\ul{dm_0h}}\ot 1)\\= \sum_{\substack{m'+m''=m_0\\(m',m'')\neq (0,0)}} (\bar\R^{-1})_{m'}^{2,3} \Psi_{\h,m''}(e^{2\ul{dm''u}}\ot q^{\mu(m'')}q^{2\ul{dm''h}}\ot 1)
\end{multline}
It is well known (see for instance~\cite{Chari1994}) that $(\bar\R^{-1})_{m'} \in \h^{|m'|}A_{\h}$, and by assumption $\Psi_{\h,m''} \in \h^{|m''|}A_{\h}$ implying that the right hand side of the above equation belongs to $\h^{|m_0|}A_{\h}$.

Let $v$ be the $\h$-adic valuation of $\Psi_{\h,m_0}$, that is $\h^{-v}\Psi_{\h,m_0} \in A_{\h}\backslash\h A_{\h}$. Assume that $v<|m_0|$ and let $\psi=(\h^{-v}\Psi_{\h,m_0} \mod \h)$ which is a non-zero element of $U(\mf b^+)\ot U(\mf n^-)[[u_1,\dots,u_r]][\frac1{\ell_{\beta}}]$, equation~\eqref{eq:hbarABRR} implies that
\[
\psi(1-e^{2\ul{dm_0u}}\ot q^{\mu(m_0)}q^{2\ul{dm_0h}}\ot 1)=\psi((1-e^{2\ul{dm_0u}}) \ot 1\ot 1) +O(\hbar) \in \h^{|m_0|-v}A_{\h}
\] 
meaning that 
\[
\psi((1-e^{2\ul{dm_0u}}) \ot 1\ot 1)=0 
\]
It leads to a contradiction because the linear map
\[
\begin{array}{rcl}
U(\mf b^+)\ot U(\mf n^-)[[u_1,\dots,u_r]][\frac1{\ell_{\beta}}]& \rightarrow &U(\mf b^+)\ot U(\mf n^-)[[u_1,\dots,u_r]][\frac1{\ell_{\beta}}] \\
x & \mapsto & x((1-e^{2\ul{dmu}}) \ot 1\ot 1)
\end{array}
\]
is injective for $m\neq 0$. Hence, $\Psi_{\h,m_0} \in \h^{|m_0|}A_{\h}$ as required.
\end{proof}
\subsubsection{The mixed pentagon equation}
\begin{thm}\label{thm:mixPentaSansShift}
The set $\{ \Psi_{q,m},m\in \N^r\}$ satisfies the mixed pentagon system of equations
\[
\forall m\in \N^r, \sum_{m=m'+m''}\Psi_{q,m'}^{1,2,34}\Psi_{q,m''}^{12,3,4}=\sum_{m=m'+m''}\Psi_{q,m'}^{1,23,4}\Psi_{q,m''}^{1,2,3}
\]
in $U_q(\mf h) \ot U_q(\mf b^+) \ot U_q(\mf g) \ot U_q(\mf n^-)[1/D^{12}, 1/D^{123}]$.
\end{thm}
\begin{proof}
Again each $\Psi_{q,m}$ is identified with its image  $\Psi_{\h,m} \in U'_{\h,loc}(\mf h)\tilde \ot U_{\h}(\mf b^+) \hat \ot U_{\h}(\mf n^-)[\h^{-1}]$. Let
\[
\Psi_{\h}=\sum_{m\in \N^r}\Psi_{\h,m} \in U'_{\h,loc}(\mf h) \tilde\ot U_{\h}(\mf b^+)\hat\ot U_{\h}(\mf n^-)
\]
(see Corollary~\ref{cor:convergence}).

The ABRR equation implies that $\Psi_{\h}^{1,2,34}$ satisfies
\[
\Psi_{\h}^{1,2,34}=(\R^{2,3})^{-1} (\R^{2,4})^{-1} q^{\tth^{2,3}/2+\tth^{2,4}/2}   q^{\tth^{1,2}+\tth^{2,2}/2}\Psi_{\h}^{1,2,34}  q^{-\tth^{1,2}-\tth^{2,2}/2} 
\]
Hence:
\[
\Psi_{\h}^{1,2,34}\Psi_{\h}^{12,3,4}=(\R^{2,3})^{-1} (\R^{2,4})^{-1} q^{\tth^{2,3}/2+\tth^{2,4}/2}   q^{\tth^{1,2}+\tth^{2,2}/2}\Psi_{\h}^{1,2,34}  \Psi_{\h}^{12,3,4}q^{-\tth^{1,2}-\tth^{2,2}/2} 
\]
Similarly, $\Psi_{\h}^{1,23,4}$ satisfies
\[
\Psi_{\h}^{1,23,4}=(\R^{3,4})^{-1} (\R^{2,4})^{-1} q^{\tth^{3,4}/2+\tth^{2,4}/2}   q^{\tth^{1,2}+\tth^{1,3}+\tth^{2,2}/2+\tth^{3,3}/2+\tth^{2,3}}\Psi_{\h}^{1,2,34} q^{-\tth^{1,2}-\tth^{1,3}-\tth^{2,2}/2-\tth^{3,3}/2-\tth^{2,3}}
\]
By the $\mf h$-invariance of $\Psi_{\h}$:
\begin{multline}\label{eq:3ABRR}
\Psi_{\h}^{1,23,4}\Psi_{\h}^{1,2,3}=(\R^{23,4})^{-1} q^{\tth^{23,4}/2}   q^{\tth^{1,2}+\tth^{1,3}+\tth^{2,2}/2+\tth^{3,3}/2+\tth^{2,3}}\\ \Psi_{\h}^{1,2,34}\Psi_{\h}^{1,2,3} q^{-\tth^{1,2}-\tth^{1,3}-\tth^{2,2}/2-\tth^{3,3}/2-\tth^{2,3}}
\end{multline}
Let $A_r,A_l$ be the linear endomorphisms of $U'_{\h,loc}(\mf h)\tilde\ot U_{\h}(\mf b^+)\hat\ot U_{\h}(\mf g)\hat\ot U_{\h}(\mf n^-)$ defined by:
\[
A_r(X)=(\R^{2,3})^{-1} (\R^{2,4})^{-1} q^{\tth^{2,3}/2+\tth^{2,4}/2}   q^{\tth^{1,2}+\tth^{2,2}/2}Xq^{-\tth^{1,2}-\tth^{2,2}/2} 
\]
and
\[
A_l(X)=(\R^{3,4})^{-1} (\R^{2,4})^{-1} q^{\tth^{3,4}/2+\tth^{2,4}/2}   q^{\tth^{1,2}+\tth^{1,3}+\tth^{2,2}/2+\tth^{3,3}/2+\tth^{2,3}}Xq^{-\tth^{1,2}-\tth^{1,3}-\tth^{2,2}/2-\tth^{3,3}/2-\tth^{2,3}}
\]

\begin{lem}
\[
A_l(\Psi_{\h}^{1,2,34}\Psi_{\h}^{12,3,4})=\Psi_{\h}^{1,2,34}\Psi_{\h}^{12,3,4}
\]
\end{lem}
\begin{proof}
It suffices to show that $A_r$ and $A_l$ commute. It reduces to show the equality:
\begin{multline}
(\R^{2,34})^{-1}  q^{\tth^{2,34}/2}   q^{\tth^{1,2}+\tth^{2,2}/2}(\R^{23,4})^{-1} q^{\tth^{23,4}/2}   q^{\tth^{1,23}+\tth^{2,2}/2+\tth^{3,3}/2+\tth^{2,3}} \\= (\R^{23,4})^{-1}  q^{\tth^{23,4}/2}   q^{\tth^{1,23}+\tth^{2,2}/2+\tth^{3,3}/2+\tth^{2,3}}(\R^{2,34})^{-1}  q^{\tth^{2,34}/2}   q^{\tth^{1,2}+\tth^{2,2}/2}
 \end{multline}
It is done by a direct computation:
\begin{align*}
 &(\R^{2,34})^{-1}  q^{\tth^{2,34}/2}   q^{\tth^{1,2}+\tth^{2,2}/2}(\R^{23,4})^{-1} q^{\tth^{23,4}/2}   q^{\tth^{1,23}+\tth^{2,2}/2+\tth^{3,3}/2+\tth^{2,3}/2+\tth^{2,3}/2}\\ 
 =&(\R^{2,34})^{-1} (\R^{3,4})^{-1} q^{\tth^{2,3}/2+\tth^{2,4}/2}   q^{\tth^{1,2}+\tth^{2,2}/2}(\R^{2,4})^{-1} q^{\tth^{2,34}/2} q^{\tth^{3,24}/2}  q^{\tth^{1,2}}q^{\tth^{1,3}}q^{\tth^{2,2}/2+\tth^{3,3}/2}\\ 
 =&(\R^{3,4})^{-1} (\R^{2,4})^{-1} (\R^{2,3})^{-1} q^{\tth^{2,3}/2+\tth^{24}/2}   q^{\tth^{1,2}+\tth^{2,2}/2+\tth^{3,3}/2}q^{\tth^{3,24}/2}q^{\tth^{1,3}}(\R^{2,4})^{-1} q^{\tth^{2,34}/2}   q^{\tth^{1,2}}q^{\tth^{2,2}/2}\\ 
 =&(\R^{3,4})^{-1} (\R^{2,4})^{-1} (\R^{2,3})^{-1} q^{\tth^{23,4}/2}   q^{\tth^{1,23}}q^{\tth^{2,2}/2+\tth^{3,3}/2+\tth^{2,3}}(\R^{2,4})^{-1} q^{\tth^{2,34}/2}   q^{\tth^{1,2}}q^{\tth^{2,2}/2}\\ 
 =&(\R^{3,4})^{-1} (\R^{2,4})^{-1} q^{\tth^{23,4}/2}   q^{\tth^{1,23}}q^{\tth^{2,2}/2+\tth^{3,3}/2+\tth^{2,3}}(\R^{2,3})^{-1} (\R^{2,4})^{-1} q^{\tth^{2,34}/2}   q^{\tth^{1,2}}q^{\tth^{2,2}/2}\\ 
\end{align*}
\end{proof}
Any $\mf h$-invariant solution of the equation $A_l(X)=X$ is uniquely determined by its degree zero part with respect to the gradation on its last component. Both $\Psi_{\h}^{1,2,34}\Psi_{\h}^{12,3,4}$ and $\Psi_{\h}^{1,23,4}\Psi_{\h}^{1,2,3}$ are $\mf h$-invariant solutions of this equation, and their degree zero parts with respect to the last component are both $\Psi_{\h}^{1,2,3}$, which implies that they are equal.
\end{proof}

\subsection{The shifted dynamical twist}\label{sec:shiftedCocycle}
Let $\nu=(\nu_1,\dots, \nu_r)\in \C^r$ and assume that for any $\beta=\sum_i \beta_i\alpha_i\in R^+$, $\prod_i \nu_i^{\beta_i} \neq 1$. Define an automorphism of $U_q(\mf g)$ by
\[
\forall i \in \{1,\dots,r\},\ \sigma_{\nu}(e_i^{\pm})=\nu_i^{\pm 1} e_i^{\pm},\quad \sigma_{\nu}(k_i^{\pm 1})=k_i^{\pm 1}
\]
For all $m\in \N^r$, define $\Psi_{q,m}^{\sigma_{\nu}}=(Sh_{\nu} \ot \id^{\ot 2})(\Psi_{q,m})$ which belongs to $U_q(\mf h)\ot U_q(\mf b^+)\ot U_q(\mf n^-)[1/D^{12}_{\nu}]$ according to Lemma~\ref{lem:morphShift}.
\begin{thm}\label{thm:mainPsi}
\begin{enumerate}[(a)]
\item The set $\{\Psi_{q,m}^{\sigma_{\nu}},m\in \N^r\}$ satisfies the system of mixed pentagon equation (see Theorem~\ref{thm:mixPentaSansShift})
\item The set $\{\Psi_{q,m}^{\sigma_{\nu}},m\in \N^r\}$ satisfies the modified ABRR system
\begin{equation} \label{eq:modifiedABRR}
\forall m\in \N^r,\Psi_{q,m}^{\sigma_{\nu}}=\sum_{m=m'+m''}(\bar \R^{-1})_{m'}^{2,3} \ul{\Ad(q^{\tth^{1,2} + \tth^{2,2}/2})}\circ (\id \ot \sigma_{\nu} \ot \id)(\Psi_{q,m''}^{\sigma_{\nu}})
\end{equation}
\end{enumerate}
\end{thm}
\begin{proof}
Part (a) follows from the fact that $\Sh_{\nu}$ commutes with the various insertion-coproduct morphisms. For proving (b) it is enough to show that the following diagram commutes:
\begin{center}
 \begin{tikzpicture}
\matrix (m) [matrix of math nodes, row sep=3em,
column sep=8em, text height=1.5ex, text depth=0.25ex]
{ U_q(\mf h)\ot U_q(\mf b^+)&   U_q(\mf h)\ot U_q(\mf b^+)\\
  U_q(\mf h)\ot U_q(\mf b^+) &  U_q(\mf h)\ot U_q(\mf b^+)\\ };
\path[->,font=\scriptsize]
(m-1-1) edge node[auto] {$\ul{\Ad(q^{\tth})}$} (m-1-2)
(m-2-1) edge node[below] {$\ul{\Ad(q^{\tth})}\circ \sigma_{\nu}^{(2)}$} (m-2-2)
(m-1-1) edge node[left] {$\Sh_{\nu}\ot \id$} (m-2-1)
(m-1-2) edge node[right] {$\Sh_{\nu}\ot \id$} (m-2-2)
;
\end{tikzpicture}
\end{center}
As the restriction of $\ul{\Ad(q^{\tth})}\circ \sigma_{\nu}^{(2)}$ and $\ul{\Ad(q^{\tth})}$ to $U_q(\mf h)^{\ot 2}$ is the identity, it suffices to prove it for the $1\ot e_i, i=1\dots r$.
Indeed, on the one hand
\begin{align*}
 \Sh_{\nu} \circ \ul{\Ad(q^{\tth})}(1\ot e_i) &= \Sh_{\nu}(k_i^2 \ot e_i)\\
&= \nu_i k_i^2 \ot e_i
\end{align*}
on the other hand
\begin{align*}
\ul{\Ad(q^{\tth})}\circ \sigma_{\nu}^{(2)} \circ \Sh_{\nu}(1\ot e_i) &= \ul{\Ad(q^{\tth})}\circ \sigma_{\nu}^{(2)}(1 \ot e_i)\\
&= \nu_i k_i^2 \ot e_i
\end{align*}
\end{proof}
Using Lemma~\ref{lem:morphShift} and by the assumption on $\nu$, $\Psi_{q,m}^{\sigma_{\nu}} \in U_q(\mf h) \ot U_q(\mf b^+) \ot U_q(\mf n^-)[1/D^{12}_{\nu}]$ can be identified with an element $(\Psi_{\h,\sigma_{\nu}})_{m} \in U_{\h}'(\mf h) \hat\ot U_{\h}(\mf b^+) \hat\ot U_{\h}(\mf n^-)[\h^{-1}]$. Let us first prove:
\begin{lem}\label{lem:shiftedConvergence}
The sum
\[
\Psi_{\h,\sigma_{\nu}}=\sum_{m\in \N^r}(\Psi_{\h,\sigma_{\nu}})_m
\]
is convergent in the $\h$-adic topology and has non-negative $\h$-adic valuation, hence defines an element of $U'_{\h}(\mf h)\hat\ot U_{\h}(\mf b^+)\hat\ot U_{\h}(\mf n^-)$of the form $1+O(\h)$.
\end{lem}
\begin{proof}
We show by induction on $|m|$ that $\forall m\in \N^r$, $(\Psi_{\h,\sigma_{\nu}})_{m} \in \h^{|m|} (U'_{\h}(\mf h)\hat\ot U_{\h}(\mf b^+)\hat\ot U_{\h}(\mf n^-))$. The modified ABRR equation leads to the relation
\begin{multline}
(\Psi_{\h,\sigma_{\nu}})_{m_0}(1-\nu^{\ul m_0}e^{2\ul{dm_0u}}\ot q^{\mu(m_0)}q^{2\ul{dm_0h}}\ot 1)\\= \sum_{\substack{m'+m''=m_0\\(m',m'')\neq (0,0)}} (\bar\R^{-1})_{m'}^{2,3} (\Psi_{\h,\sigma_{\nu}})_{m''}(\nu^{\ul m''}e^{2\ul{dm''u}}\ot q^{\mu(m'')}q^{2\ul{dm''h}}\ot 1)
\end{multline}
Then, as in the proof of Corollary~\ref{cor:convergence} one shows that the $\h$-adic valuation of $(\Psi_{\h,\sigma})_{m}$ is at least $|m|$ by remarking that the map
\[
\begin{array}{rcl}
U(\mf b^+)\ot U(\mf n^-)[[u_1,\dots,u_r]]& \rightarrow & U(\mf b^+)\ot U(\mf n^-)[[u_1,\dots,u_r]]\\
x & \mapsto & x((1-\nu^{\ul m}e^{2\ul{dmu}}) \ot 1\ot 1)
\end{array}
\]
is injective for $m\neq 0$. The last statement follows from the fact that $(\Psi_{\h,\sigma_{\nu}})_0=1$.

\end{proof}
We finally get the desired result:
\begin{thm}\label{thm:formalPentaABRR}
The element $\Psi_{\h,\sigma_{\nu}}$ satisfies the mixed pentagon equation 
\[
(\Psi_{\h,\sigma_{\nu}})^{1,2,34}(\Psi_{\h,\sigma_{\nu}})^{12,3,4}=(\Psi_{\h,\sigma_{\nu}})^{1,23,4}(\Psi_{\h,\sigma_{\nu}})^{1,2,3} 
\]
in $U_{\h}(\mf h)\hat \ot U_{\h}(\mf g)^{\hat\ot 3}$ and the modified ABRR equation
\[
\Psi_{\h,\sigma_{\nu}}= (\bar\R^{-1})^{2,3} \Ad(q^{\tth^{1,2} + \tth^{2,2}/2})\circ (\id \ot \sigma_{\nu} \ot \id)(\Psi_{\h,\sigma_{\nu}})
\]
in $U_{\h}(\mf h)\hat\ot U_{\h}(\mf g)^{\hat \ot 2}$.
\end{thm}
\begin{proof}
For a given $m \in \N^r$,  
\[
\sum_{m'+m''=m}(\Psi_{\h,\sigma_{\nu}})_{m'}^{1,2,34}(\Psi_{\h,\sigma_{\nu}})_{m''}^{12,3,4} 
\]
and
\[
\sum_{m'+m''=m}(\Psi_{\h,\sigma_{\nu}})_{m'}^{1,23,4}(\Psi_{\h,\sigma_{\nu}})_{m''}^{1,2,3} 
\]
are equal according to Theorem~\ref{thm:mainPsi}, and both belong to $\h^{|m|} U'_{\h}(\mf h)\hat\ot U_{\h}(\mf b^+)\hat\ot U_{\h}(\mf n^-)$. Hence the sums over $m$ of these expressions are convergent in the $\h$-adic topology and:
\[
\sum_{m\in \N^r}\sum_{m'+m''=m}(\Psi_{\h,\sigma_{\nu}})_{m'}^{1,2,34}(\Psi_{\h,\sigma_{\nu}})_{m''}^{12,3,4}=\sum_{m\in \N^r}\sum_{m'+m''=m}(\Psi_{\h,\sigma_{\nu}})_{m'}^{1,23,4}(\Psi_{\h,\sigma_{\nu}})_{m''}^{1,2,3} 
\]
which implies the mixed pentagon equation.
A similar argument implies the second statement.
\end{proof}
\section{An algebraic QRA over $U_{\h}(\mf g)$}
In this section, we prove general results relating the ABRR equation and the octagon equation (Theorem~\ref{thm:OctABRR}). Using the objects constructed in the previous section, this gives rise to an ``algebraic'' QRA (Theorem~\ref{thm:alg-QRA}). We then give an explicit form of the representations of $B_n^1$ attached to this QRA.
\subsection{Relationship between ABRR and the octagon equation}
Let $\nu=(\nu_1,\dots,\nu_r) \in (\C^{\times})^r$, $\sigma=\sigma_{\nu}$ (as in section~\ref{sec:shiftedCocycle}). Let $A_{\alg}=U_{\h}(\mf g)\rtimes_{\sigma} \Z$ (the generator of $\Z$ is denoted by $\tilde \sigma)$ be equipped with the coproduct $\Delta_{\h}$ extended to $A_{\alg}$ by setting $\Delta_{\h}(\tilde\sigma)=\tilde\sigma \ot \tilde\sigma$. Let $B_{\alg}=U_{\h}(\mf h) \subset A_{\alg}$ and set $\Delta_{B_{\alg}}=(\Delta_{\h})_{\vert B_{\alg}}$. Let $\R \in U_{\h}(\mf g)^{\hat\ot 2}$ be the $R$-matrix of $U_{\h}(\mf g)$ and $K=q^{t_{\mf h}/2}$. Recall that $\bar\R=K^{-1}\R \in U_{\h}(\mf n^+)\hat\ot U_{\h}(\mf n^-)$. 
\begin{prop}
Let $\Psi \in (B_{\alg} \hat\ot U_{\hbar}(\mf b^+) \hat\ot U_{\hbar}(\mf n^-))^{\times}$ and $E \in (U_{\hbar}(\mf h)\hat\ot (U_{\h}(\mf h)\rtimes_{\sigma}\Z))^{\times}$. Then the octagon equation~\eqref{eq:octo} for $(E,\Psi,\R)$ is equivalent to the following system of equations:
 \begin{subequations}
 \label{eq:SystOct}
 \begin{align}
 \bar\R^{3,2}  \Psi^{1,3,2} E^{1,3}&=E^{1,3}\Psi^{1,3,2} \label{eq:SystOcta}\\
   \Psi E^{12,3}&=E^{12,3}\bar\R^{2,3}\Psi \label{eq:SystOctb}\\
   E^{12,3}&=(K^{2,3})^2 E^{1,3}\label{eq:SystOctc}
  \end{align}
 \end{subequations}
\end{prop}
\begin{proof}
Rewrite the octagon equation as
 \[
  \Psi E^{12,3} \Psi^{-1}(\R^{2,3})^{-1}= \R^{3,2}\Psi^{1,3,2}E^{1,3} (\Psi^{1,3,2})^{-1}
 \]
The left-hand side belongs to $U_{\hbar}(\mf h)\hat\ot U_{\hbar}(\mf b^+)\hat\ot (U_{\hbar}(\mf b^-)\rtimes_{\sigma}\Z)$, and the right-hand side to $U_{\hbar}(\mf h)\hat\ot U_{\hbar}(\mf b^-) \hat\ot (U_{\hbar}(\mf b^+) \rtimes_{\sigma}\Z)$. Thus, each side has to coincide with its degree (0,0) part, and their degree (0,0) parts have to be equal as well (here $U_{\h}(\mf b^{\pm})\rtimes_{\sigma}\Z$ is graded by $|e_i^{\pm}|=\delta_i, |h_i|=|\tilde\sigma|=0$). Denoting by $\tilde\Psi$ the image of $\Psi$ under this projection, the octagon equation is equivalent to:
\begin{align*}
\Psi E^{12,3} \Psi^{-1}(\bar\R^{2,3})^{-1}(K^{2,3})^{-1}&=\tilde\Psi E^{12,3} \tilde\Psi^{-1}(K^{2,3})^{-1}\\
K^{2,3}\bar\R^{3,2}\Psi^{1,3,2}E^{1,3} (\Psi^{1,3,2})^{-1}&=K^{2,3}\tilde\Psi^{1,3,2}E^{1,3} (\tilde\Psi^{1,3,2})^{-1}\\
\tilde\Psi E^{12,3} \tilde\Psi^{-1}(K^{2,3})^{-1}&=K^{2,3}\tilde\Psi^{1,3,2}E^{1,3} (\tilde\Psi^{1,3,2})^{-1}
\end{align*}
Then, the commutativity of $U_{\h}(\mf h)\rtimes_{\sigma}\Z$ implies that $\tilde\Psi$ and $\tilde\Psi^{-1}$ cancel out, which leads to the system~\eqref{eq:SystOct}.
\end{proof}
We make the system~\eqref{eq:SystOct} explicit in a particular case.
\begin{thm}
\label{thm:OctABRR}
 Set $E_{\h,\sigma}=q^{ t_{\mf h}^{1,2}+\frac12t_{\mf h}^{2,2}} (1 \ot \tilde\sigma)$. For any solution $\Psi \in (U_{\hbar}'(\mf h) \ot U_{\hbar}(\mf b^+) \ot U_{\hbar}(\mf b^-))^{\mf h}$ of the linear equation
 \begin{equation}
 \label{eq:algABRR}
   \Psi^{1,2,3} (E^{\sigma}_{\h})^{1,2}=(\bar\R^{2,3})^{-1}(E^{\sigma}_{\h})^{1,2}\Psi^{1,2,3},
 \end{equation}
the triple $(E_{\h,\sigma}, \Psi,\R)$ satisfies the octagon equation.
\end{thm}
\begin{proof}
Let us first check~\eqref{eq:SystOctc}:
\begin{align*}
 (\Delta_{\h} \ot \id)(q^{ t_{\mf h}^{1,2}+\frac12t_{\mf h}^{2,2}} (1 \ot \tilde\sigma)) &= q^{ t_{\mf h}^{1,3}+t_{\mf h}^{2,3}+\frac12t_{\mf h}^{3,3}} (1 \ot 1\ot \tilde\sigma) \\
 &=q^{t_{\mf h}^{2,3} } q^{ t_{\mf h}^{1,3}+\frac12t_{\mf h}^{3,3}} (1 \ot 1\ot \tilde\sigma)\\
 &= (K^{2,3})^2 (E_{\h,\sigma})^{1,3}.
\end{align*}

Recall from equation~\eqref{eq:temp}, section~\ref{sec:resultRep} that $E_{\h,\sigma}$ satisfies:
\begin{equation}\label{eq:ansatz}
 (E_{\h,\sigma})^{1,23}= (E_{\h,\sigma})^{1,2}(E_{\h,\sigma})^{1,3}(K^{2,3})^2
 \end{equation}
Equation~\eqref{eq:algABRR} is the same as~\eqref{eq:SystOcta}, thus, starting from~\eqref{eq:SystOcta}, permuting the two last component and multiplying by $(E_{\h,\sigma})^{1,3}(K^{2,3})^2$ on the right, one gets
\begin{align*}
 \bar\R^{2,3}  \Psi^{1,2,3} (E_{\h,\sigma})^{1,2}(E_{\h,\sigma})^{1,3}(K^{2,3})^2&=(E_{\h,\sigma})^{1,2}\Psi^{1,2,3}(E_{\h,\sigma})^{1,3}(K^{2,3})^2 
\end{align*}
Then, using equation~\eqref{eq:temp}
\begin{align}\label{eq:tempA}
\bar\R^{2,3} \Psi (E_{\h,\sigma})^{1,23} & = (E_{\h,\sigma})^{1,2}  \Psi (E_{\h,\sigma})^{1,3}(K^{2,3})^2 
 \end{align}
By the $\mf h$-invariance of $\bar\R$ and $\Psi$ and the commutativity of $U_{\h}(\mf h)$, $\bar\R^{2,3}\Psi$ commutes with $(E_{\h,\sigma})^{1,23}$, and $(E_{\h,\sigma})^{1,3}$ commutes with $(K^{2,3})^2$. Using Equation~\eqref{eq:SystOctc}, equation~\eqref{eq:tempA} then implies
 \begin{align}\label{eq:tempB}
(E_{\h,\sigma})^{1,23} \bar\R^{2,3}\Psi &=(E_{\h,\sigma})^{1,2} \Psi (E_{\h,\sigma})^{12,3}
\end{align}
Equations~\eqref{eq:ansatz} and~\eqref{eq:SystOctc} together with the commutativity of $U_{\h}(\mf h)$ implies that $(E_{\h,\sigma})^{1,23}=(E_{\h,\sigma})^{1,2}(E_{\h,\sigma})^{12,3}$. Equation~\eqref{eq:tempB} then implies 
\[
 (E_{\h,\sigma})^{1,2} (E_{\h,\sigma})^{12,3} \bar\R^{2,3}\Psi =(E_{\h,\sigma})^{1,2} \Psi (E_{\h,\sigma})^{12,3}
\]
which implies
\[
(E_{\h,\sigma})^{12,3} \bar\R^{2,3}\Psi = \Psi (E_{\h,\sigma})^{12,3}
\]
The last line is~\eqref{eq:SystOctb}. Then, by Proposition~\ref{thm:OctABRR}, $(E_{\h,\sigma},\Psi,\R)$ satisfies the octagon equation.
\end{proof}
\subsection{A QRA over $U_{\h}(\mf g)$}
We are now in a position to state the main Theorem of this section: 
\begin{thm}\label{thm:alg-QRA}
 $B_{\alg,\sigma}=(U_{\h}(\mf h), \Delta_{\h}, \Psi_{\h,\sigma}, E_{\h,\sigma})$ is a QRA over $A_{\alg}=U_{\h}(\mf g) \rtimes_{\sigma} \Z $.
\end{thm}
\begin{proof}
The relation~\eqref{eq:Ecop} is clear  and the relation~\eqref{eq:PsiCop} is the $\mf h$-invariance of $\Psi_{\h,\sigma}$. The element $\Psi_{\h,\sigma}$ belongs to $U_{\h}(\mf h)\hat\ot U_{\h}(\mf g)^{\hat\ot 2}$ and satisfies the mixed pentagon equation~\eqref{eq:mixedPenta} with trivial associator according to Theorem~\ref{thm:formalPentaABRR}. Equation~\eqref{eq:algABRR} of Theorem~\ref{thm:OctABRR} is nothing but the ABRR equation~\eqref{eq:modifiedABRR} of Theorem~\ref{thm:formalPentaABRR}, meaning that $\Psi_{\h,\sigma}$ satisfies the octagon equation with $(E_{\h,\sigma},\R)$.
\end{proof}
\subsection{An explicit formula for the representations of $B_n^1$}
Being a QRA, $B_{\alg,\sigma}$ induces a group morphism $\rho_{\h}:B_n^1 \rightarrow (B_{\alg,\sigma} \ot A_{\alg}^{\ot n})^{\times}$.
\begin{thm}
 The morphism $\rho_{\h}$ coincides with the morphism given in Theorem~\ref{thm:rep_gtsig}.
\end{thm}
\begin{proof}
 According to Section~\ref{sec:QRA}, the image of the generator $\tau$ is
 \[
  (\Psi_{\h,\sigma}^{0,1,2\dots n}) E_{\h,\sigma}^{0,1}(\Psi_{\h,\sigma}^{0,1,2\dots n})^{-1}
 \]
that is
\[
 (\id_{U_{\h}(\mf h)} \ot \id_{U_{\h}(\mf g)} \ot \Delta_{\h}^{(n-1)}(\Psi_{\h,\sigma} E_{\h,\sigma}^{0,1}\Psi_{\h,\sigma}^{-1} )
\]
Thanks to~\eqref{eq:algABRR}, this is equal to
\[
 (\id_{U_{\h}(\mf h)} \ot \id_{U_{\h}(\mf g)} \ot \Delta_{\h}^{(n-1)}) ((\R^{1,2})^{-1}K^{1,2}E_{\h,\sigma}^{0,1})
\]
We then use the fact the $K^{1,23}=K^{1,2}K^{1,3}$ and $\R^{1,23}=\R^{1,3}\R^{1,2}$ to identify this with the image of $\tau$ in Theorem~\ref{thm:rep_gtsig}.
\end{proof}
\section{A QRA arising from KZ equations}
The monodromy morphism of the cyclotomic KZ differential system can be expressed algebraically by using cyclotomic analogs of the Drinfeld's associator~\cite{Enriquez2008}. Recall that the Drinfeld KZ associator $\Phi_{\kz} \in (U(\mf g)^{\ot 3}[[\h]])^{\mf g}$ is defined as the renormalized holonomy from 0 to 1 of the differential equation
\begin{equation}
 \frac{\dd}{\dd z} G(z) = \frac{\h}{2\pi\sqrt{-1}} \left(\frac{t^{1,2}}{z} + \frac{t^{2,3}}{1-z} \right) G(z)
\end{equation}
i.e. $\Phi_{\kz}=G_0^{-1}G_1$ where $G_0,\ G_1$ are the solutions such that $G_0 \sim z^{\h t^{1,2}}$ when $z \rightarrow 0^+$ on the real axis and $G_1 \sim (1-z)^{\h t^{2,3}}$ when $z \rightarrow 1^-$. Here,  $G_0 \sim z^{t^{1,2}}$ means that  $\tilde G(z)=G_0(z) z^{-\h t^{1,2}}$ is analytic on the interval $]-1,1[$ and $\tilde G(0)=1$.

In the same way, The KZ dynamical pseudo-twist $\Psi_{\kz,\sigma}$ is defined as the renormalized holonomy from 0 to 1 of the following differential equation:
\begin{equation}
 \frac{\dd}{\dd z} H(z) = \frac{\h}{2\pi\sqrt{-1}}\left( \frac{N(t_{\mf h}^{0,1} + \frac12 t_{\mf h}^{1,1})}{z} + \sum_{ a\in \Z/N\Z} \frac{ (\sigma^a \ot \id) t^{1,2}}{z-\zeta_N^a} \right) H(z)
\end{equation}
It means that $\Psi_{\kz,\sigma}=H_1^{-1} H_0$ where $H_0,\ H_1$ are the solutions such that $H_0(z) \sim z^{\h N(t_{\mf h}^{0,1} + \frac12 t_{\mf h}^{1,1})}$ when $z \rightarrow 0^+$ and $H_1(z) \sim  z^{\h t^{1,2}}$ when $z \rightarrow 1^-$. Thus, $\Psi_{\kz,\sigma} \in (U(\mf h) \ot U(\mf g)^{\ot 2} [[\h]])^{\mf h}$.

These differential equations comes, up to a change of variable,from the original KZ system for $n=3$ and from the cyclotomic KZ system for $n=2$ respectively. Finally, the following theorems implies that these monodromy representations can be expressed in the algebraic/categorical language:

\begin{thm}[Drinfeld]
 $A_{\kz}=(U(\mf g)[[\hbar]], \Delta_0, R=e^{\hbar t/2}, \Phi_{\kz})$ is a QTQBA.
\end{thm}
\begin{thm}[Enriquez]
 $B_{\kz,\sigma}=(U(\mf h)[[\hbar]], \Delta_0, E_{\kz,\sigma}=e^{\hbar (t_{\mf h}^{1,2} + \frac12 t_{\mf h}^{2,2})}(1 \ot \sigma), \Psi_{\kz,\sigma})$ is a QRA over $A_{\kz}\rtimes Z_{\sigma}$, whose induced representation of $B_n^1$ are exactly those of the monodromy of the cyclotomic KZ system.
\end{thm}
\section{Equivalence of representations}\label{sec:equiv}
Let $U(\mf h)[[\h]'$ be the QFSHA~\cite{Gavarini2002} associated to the Quantized Enveloping Algebra $U(\mf h)[[\h]]$. Under the identification $U(\mf h)\cong\C[h_1,\dots,h_r]$, $U(\mf h)[[\h]]' \cong \C[[\h h_1,\dots, \h h_r,\h]]\subset \C[h_1,\dots,h_ r][[\h]]$.

We have constructed two QRA's:
\begin{enumerate}[(i)]
\item $B_{\alg}=(U_{\h}(\mf h),\Delta_{\h},\Psi_{\h,\sigma},E_{\h,\sigma})$ over $A_{\alg}=(U_{\h}(\mf g)\rtimes_{\sigma_{\h}}\Z,\Delta_{\h},\Phi=1,\R)$, where $\Psi_{\h,\sigma}$ belongs to $U'_{\h}(\mf h)\hat \ot U_{\h}(\mf g)^{\hat\ot 2}$
\item $B_{\kz}=(U(\mf h)[[\h]],\Delta_0,\Psi_{\kz,\sigma},E_{\kz,\sigma})$ over $A_{\kz}=(U(\mf g)[[\h]]\rtimes_{\sigma}\Z,\Delta_0,\Phi_{\kz},\R_{\kz})$, where $\Psi_{\kz,\sigma}$ belongs to $U(\mf h)[[\h]]' \hat\ot U(\mf g)^{\hat\ot 2}[[\h]]$ according to~\cite[Prop. 4.7]{Enriquez2005}.
\end{enumerate}
\begin{thm}\label{thm:main}
These two QRA's are twist equivalent.
\end{thm}

Let $V$ be a $\mf g$-module, $V_{\mf h}$ the corresponding $U_{\hbar}(\mf g)$-module, $W$ an $\mf h$-module and $W_{\hbar}=W[[\h]]$ the corresponding $U_{\hbar}(\mf h)$-module. From the data of $(\mf g, t,\sigma)$, one can construct:
\begin{itemize}
 \item a morphism $\rho_{\kz}:B_n^1 \rightarrow GL(W\ot V^{\ot n}[[\hbar]])$ using the KZ equation
 \item a morphism $\rho_{\h}:B_n^1 \rightarrow GL(W_{\hbar}\ot V_{\h}^{\ot n})$ using the QRA $B_{\alg}$.
\end{itemize}

\begin{cor}
The representations $\rho_{\kz}$ and $\rho_{\h}$ are equivalent.
\end{cor}

The rest of this section is devoted to the proof of this Theorem. We first apply a twist bringing $A_{\alg}$ to $A_{\kz}$, the subalgebra $B_{\alg}$ to $B_ {\kz}$ and $E_{\alg}$ to $E_{\kz}$. Then, both $\Psi=\Psi_{\kz,\sigma}$ and the image $\tilde\Psi$ of $\Psi_{\h,\sigma}$ satisfy the mixed pentagon equation with $\Phi_{\kz}$. Using deformation theory arguments, we then prove that $\Psi$ and $\tilde\Psi$ are related by an infinitesimal functional shift and an $\mf h$-invariant twist. The fact that both $\Psi$ and the twisted $\tilde\Psi$ satisfy the octagon equation with $E_{\kz}$ and $\R_{\kz}$ implies that the shift is actually trivial.

Let us first recall the following
 \begin{thm}[\cite{Drinfeld1990a,Drinfeld1990}]
 \label{thm:DriKD}
  There exists an algebra isomorphism $\alpha:U_{\hbar}(\mf g) \rightarrow U(\mf g)[[\hbar]]$ and a twist $F \in U(\mf g)^{\ot 2}[[\hbar]]$ such that:
  \begin{enumerate}[(a)]
   \item $(U(\mf g)[[\h]],\Delta_0,\Phi_{\kz},\R_{\kz})=\alpha((U_{\h}(\mf g),\Delta_{\h},1,\R))^F$ as QTQHA
   \item $\alpha$ restricts to the canonical isomorphism $U_{\hbar}(\mf h)\rightarrow U(\mf h)[[\hbar]]$
   \item $F$ is $\mf h$-invariant
  \end{enumerate}
 \end{thm}
\begin{cor}
The isomorphism $\alpha$ extends to an algebra isomorphism $\alpha:A_{\alg} \rightarrow A_{\kz}$ and $A_{\kz}=\alpha(A_{\alg})^F$ as QTQHA.
\end{cor}
\begin{proof}
Set $\alpha(\tilde\sigma_{\h})=\tilde \sigma$. Both $U_{\h}(\mf g)$ and $U(\mf g)[[\h]]$ are direct sums of their weight subspaces, and $\alpha$ preserves these decompositions thanks to property (b). Since $\sigma_{\h}$ acts on a given weight subspace of $U_{\h}(\mf g)$ as $\sigma$ does on the corresponding weight subspace of $U(\mf g)[[\h]]$, $\alpha$ intertwines the action of these automorphisms. This implies the first statement.

Being $\mf h$-invariant, $F$ belongs to $(\bigoplus_{m\in \Z^r} U(\mf g)[m]\ot U(\mf g)[-m])[[\h]]$, hence $(\sigma \ot \sigma)(F)=F$. Therefore,
\begin{align*}
 F\left(\alpha^{\ot 2}\circ\Delta_{\h}\circ \alpha^{-1}(\tilde\sigma)\right)F^{-1} &= F (\tilde\sigma \ot \tilde\sigma)F^{-1}\\
&= F\sigma^{\ot 2}(F^{-1}) (\tilde\sigma \ot \tilde\sigma)\\
&=\Delta_0(\tilde\sigma)
\end{align*}
which implies the second statement.
\end{proof}
We then apply the twist $(F,1)$ to the QRA $B_{\alg}$. The restriction of $\alpha$ leads to the algebra isomorphism $U_{\h}(\mf h) \rightarrow U(\mf h)[[\h]]$ given by $h_i \mapsto h_i$. Hence, it maps the subalgebra $U'_{\h}(\mf h)$ to $U(\mf h)[[\h]]'$ and $E_{\h,\sigma}=q^{\tth+\frac12\tth^{2,2}}$ to $E_{\kz}=e^{\h(\tth+\frac12 \tth^{2,2})}$. It leads to a QRA
\begin{enumerate}[(iii)]
\item $B=(U(\mf h)[[\h]],\Delta_0, E_{\kz},\Psi')$ over $A_{\kz}$
\end{enumerate}

Hence, it remains to prove that $\Psi$ and $\Psi'$ are related by a twist. We first show that they are related by a twist and a ``shift'' (Theorem~\ref{thm:shift}), and then that the shift actually vanishes (Proposition~\ref{prop:rigid}). 
\subsection{Gauge transformation and shifts}
Following~\cite{Enriquez2003}, let $\cg$ be the subgroup of $(U(\mf h)[[\h]]' \hat\ot U(\mf g)^{\mf h})^{\times}$ of elements of the form
\[
G=1+\sum_{n\geq 1} \h^n g_n, g_n \in U(\mf h) \ot U(\mf g)
\]
sucht that the image of $G$ through the tensor product of the reudction maps $U(\mf h)[[\h]]'\rightarrow \widehat{S(\mf h)}$ and $U(\mf g)[[\h]]\rightarrow U(\mf g)$ is of the form $\exp(q)$ where $q\in \widehat{S(\mf h)}\ot \mf g$. If $\Psi \in 1+\h (U(\mf h)[[\h]]' \hat \ot U(\mf g)^{\ot 2}[[\h]])^{\mf h}$ is a solution of the mixed pentagon equation, define the \emph{gauge transformation} of $\Psi$ by $G$ by
\[
G\star \Psi:= G^{1,23} \Psi (G^{12,3})^{-1} (G^{1,2})^{-1}.
\]
Then, $G\star \Psi$ belongs to $1+\h (U(\mf h)[[\h]]' \hat \ot U(\mf g)^{\ot 2}[[\h]])^{\mf h}$ and satisfies the mixed pentagon equation. Indeed, a gauge transformation is a particular case of the notion of twist of a QRA as defined in section~\ref{sec:QRA}.

The notion of shift is defined in the following way: for $\mu \in \mf h^*[[\h]]$, define an automorphism of $U(\mf h)[[\h]]$ by $h_i \mapsto h_i + \mu(h_i)$. This restricts to an automorphism $\sh_{\mu}$ of $U(\mf h)[[\h]]'= \C[[u_1,\dots,u_r,\h]])$ (we set $u_i=\h h_i$) given by $u_i \mapsto u_i+\h\mu(h_i)$. For $X \in U(\mf h)[[\h]]' \hat \ot U(\mf g)^{\ot n}[[\h]]$ set 
\[
X_{\mu}=(\sh_{\mu} \ot \id^{\ot n})(X)
\]
The action of $\mf h^*[[\h]]$ restricts to an action on $\cg$, and if $G \in \cg$ and $\mu \in \mf h^*[[\h]]$, one has
\[
(G^{1,23} \Psi (G^{12,3})^{-1} (G^{1,2})^{-1})_{\mu}=G_{\mu}^{1,23} \Psi_{\mu} (G_{\mu}^{12,3})^{-1} (G_{\mu}^{1,2})^{-1}.
\]
Hence, there is a well defined action of $\cg \rtimes (\mf h^*[[\h]])$ on the set $1+\h (U(\mf h)[[\h]]'\hat \ot U(\mf g)^{\ot 2}[[\h]])^{\mf h}$.
\subsection{Classification of dynamical pseudo twists}
Let $\mf m =\mf n^+ \oplus \mf n^-\subset \mf g$,  and for any $\alpha \in R^+$, choose basis elements $e_{\pm\alpha}$ of the root subspaces $\mf g_{\pm\alpha}$ such that $(e_{\alpha},e_{-\alpha})=1$. Hence,
\[
t_{\mm}=t-t_{\mf h}=2\sum_{\alpha \in R^+} e_{\alpha}\ot e_{-\alpha}+e_{-\alpha}\ot e_{\alpha} \in S^2(\mm)
\].
In this section, we will prove the following:
\begin{thm}
\label{thm:shift}
 Let $\Psi,\Psi' \in 1+\hbar (U(\mf h)[[h]]'\hat\ot U(\mf g)^{\ot 2}[[\hbar]])^{\mf h}$ be two solutions of the mixed pentagon equation
\[
 \Psi^{1,2,34} \Psi^{12,3,4}=\Phi_{\kz}^{2,3,4} \Psi^{1,23,4} \Psi^{1,2,3},
\]
such that:
\begin{enumerate}
\item the image of $\frac{\Psi-1}{\h}$ through the reduction map
\[
U(\mf h)[[\h]]' \hat \ot U(\mf g)^{\ot 2}[[\h]]\subset U(\mf h)\ot U(\mf g)^{\ot 2}[[\h]]\longrightarrow U(\mf h)\ot U(\mf g)^{\ot 2}
\]

is of the form $1\ot \rho_0$ where $\rho_0 \in  (\wedge^2 \mf m )^{\mf h}$.
\item if $\rho_0=\sum_{\alpha \in R^+}\rho_{\alpha} (e_{\alpha} \wedge e_{-\alpha})$, where $\rho_{\alpha} \in \C$ then for any $\alpha \in R^+$, $\rho_{\alpha}\neq \pm\frac12$.
\item 
\[
\Psi-\Psi' \in \h^2 U(\mf h)\ot U(\mf g)^{\ot 2}[[\h]]
\]
\end{enumerate}
Then there exists $(G,\mu) \in \cg \rtimes \mf h^*[[\h]]$ such that the relation
\[
\Psi'= G\star \Psi_{\mu}
\]
holds in $U(\mf h)[[\h]]' \hat \ot U(\mf g)^{\ot 2}[[\h]]$. 
\end{thm}
\begin{proof}

Let us explain the strategy of the proof: we will first prove the following:
\begin{prop}\label{prop:main}
If $\Psi,\Psi' \in 1+\h (U(\mf h)[[\h]]'\hat\ot U(\mf g)^{\ot 2}[[\h]])^{\mf h}$ are two solutions of the mixed pentagon equation such that for some $n\geq 2$
\[
\Psi-\Psi' \in \h^n (U(\mf h)\ot U(\mf g)^{\ot 2}[[\h]])
\]
then there exists $(G,\mu) \in \cg\rtimes \mf h^*[[\h]]$ such that $G-1\in \h^n U(\mf h)\ot U(\mf g)[[\h]]$, $\mu \in \h^{n-2}\mf h^*[[\h]]$ and
\[
\Psi'-G\star\Psi_{\mu} \in \h^{n+1} (U(\mf h)\ot U(\mf g)^{\ot 2}[[\h]]).
\]
\end{prop}

Then we construct an infinite sequence $(G_1,\mu_1),(G_2,\mu_2),\dots$, $G_n \in 1+\h^{n+1}(U(\mf h)\ot U(\mf g)[[\h]])^{\mf h}$, $\mu_n \in \h^{n-1} \mf h^*[[\h]] $, such that if $\Psi^{(0)}=\Psi$ and
\[
 \Psi^{(n)}=G_n\star\Psi^{(n-1)}_{\mu_n}
\]
then $\Psi^{(n)} \rightarrow \Psi'$ in the $\h$-adic topology of $U(\mf h)\ot U(\mf g)^{\ot 2}[[\h]]$. We construct this sequence recursively by applying Proposition~\ref{prop:main} to the pair $(\Psi^{(n-1)},\Psi')$.

Hence, thanks to the properties of the $(G_n,\mu_n)$, the product
\[
\bar G=\prod^>_{n\geq 1}(G_n)_{\sum_{k>n}\mu_k}
\]
converges in $U(\mf h)\ot U(\mf g)[[\h]]$, belongs to $U(\mf h)[[\h]]' \hat \ot U(\mf g)[[\h]]$ (since this is a closed subspace of the latter for the $\h$-adic topology) and is invertible and $\mf h$-invariant, and the sum
\[
\bar\mu=\sum_{k\geq 1}\mu_k
\]
converges in $\mf h^*[[\h]]$. Finally, a direct computation shows that
\[
 \Psi'=\bar G\star\Psi_{\bar\mu}
\]
implying the conclusion of the Theorem.

Let us recall some facts about the cohomological structure associated to the mixed pentagon equation. We define the cochain complex
\[
 C_1 =\bigoplus_{n\geq 0} (U(\hh)\ot U(\mf g)^{\ot n})^{\hh}
\]
with differential
\[
 \begin{array}{rccl}
 d_1^{n,n+1}:&U(\mf h) \ot U(\mf g)^{\ot n}& \longrightarrow & U(\mf h) \ot U(\mf g)^{\ot n+1} \\
 &x&\longmapsto & x^{1,2,\dots,n+1}+\displaystyle{\sum_{i=1}^{n+1}} (-1)^i x^{1,2,\dots,ii+1,\dots,n+2}
\end{array}
\]
Let $C_2=\bigoplus (S(\mm)^{\ot n})^{\hh}$ together with the usual coHochschild differential
\begin{equation}\label{eq:coHochschild}
 \begin{array}{rccl}
 d_2^{n,n+1}:&S(\mf m)^{\ot n}& \longrightarrow & S(\mf m)^{\ot n+1}\\
 &x&\longmapsto & x^{1,2,\dots,n}+\displaystyle{\sum_{i=1}^{n}} (-1)^i x^{1,2,\dots,ii+1,\dots,n+1} + (-1)^{n+1} x^{2,3,\dots,n+1}
\end{array}
\end{equation}

The inclusion $\mf m \subset \mf g$ extend to an inclusion $S(\mf m) \subset S(\mf g)$. Recall that $S(\mf g)$ and $U(\mf g)$ are isomorphic as coalgebras. Thus, as $d_2$ only involves the coalgebra structures, $(C_2,d_2)$ can be embedded into $(C_1,d_1)$. The cohomology of the complex $(C_1,d_1)$ was computed by D. Calaque in the general case of a reductive pair (which is the case here):
\begin{thm}[\cite{Calaque2006}]\label{thm:damien}
 \begin{enumerate}
 \item There exists an $\hh$-equivariant projection $P:U(\mf g) \rightarrow S(\mm)$ restricting to the projection $\mf g \rightarrow \mf m$ along $\mf h$ and such that
\[
 \epsilon \ot P^{\ot n}:(U(\mf h) \ot U(\mf g)^{\ot n})^{\mf h} \longrightarrow (S(\mf m)^{\ot n})^{\mf h}
\]
induces an isomorphism at the level of cohomology.
\item Moreover, there exists a linear map $\kappa:C_1 \rightarrow C_1[-1]$ such that:
\begin{itemize}
\item $\kappa d_1+d_1\kappa = \id - (\epsilon \ot P^{\ot n})$
\item $\kappa(U(\mf h)_{\leq k} \ot U(\mf g)^{\ot n})\subset U(\mf h)_{\leq k+1} \ot U(\mf g)^{\ot n-1}$
\end{itemize}
\end{enumerate}
\end{thm}
Here $U(\mf h)_{\leq n}$ is the subspace of $U(\mf h)$ of elements of degree at most $n$ in the generators $h_1,\dots,h_r$.
\begin{cor}
 The $n$th cohomology group $H^n(C_1, d_1)$ is isomorphic to $(\wedge^n \mm)^{\hh}$.
\end{cor}

If $V$ is any vector space, define the operator $\alt_n:V^{\ot n}\rightarrow \wedge^n V \subset V^{\ot n}$ by
\[
\alt_n(v_1 \ot \dots \ot v_n ) = \sum_{\sigma \in S_n} \operatorname{sgn}(\sigma) (v_{\sigma(1)} \ot \dots \ot v_{\sigma(n)})
\]
The following Corollary will also be useful:
\begin{cor}
\label{cor:cobord}
 The linear map $\alt_n \circ (\epsilon\ot P^{\ot n}):U(\mf h) \ot U(\mf g)^{\ot n}\rightarrow S(\mf m)^{\ot n}$ maps $\ker d_1^{n,n+1}$ onto $(\wedge^n \mf m)^{\mf h}$ and $\im d_1^{n-1,n}$ onto 0.
\end{cor}
\begin{proof}
 According to~\cite{Drinfeld1990} the $\alt_n$ operator induces a linear isomorphism 
\[
 \alt_n:\frac{\ker d_2^{n,n+1}}{\im d_2^{n-1,n}} \longrightarrow (\wedge^n \mf m)^{\mf h}
\]
and according to~Theorem~\ref{thm:damien}, the map $\epsilon \ot P^{\ot n}$ induces a linear isomorphism
\[
\epsilon\ot P^{\ot n} : \frac{\ker d_1^{n,n+1}}{\im d_1^{n-1,n}}  \longrightarrow \frac{\ker d_2^{n,n+1}}{\im d_2^{n-1,n}} 
\]
Hence, the composition of these two maps is a linear isomorphism
\[
\frac{\ker d_1^{n,n+1}}{\im d_1^{n-1,n}}  \longrightarrow (\wedge^n \mf m)^{\mf h}
\]
\end{proof}

Let $\cyb(x,y)$ be the bilinear map $\mf g^{\ot 2} \times \mf g^{\ot 2} \rightarrow \mf g^{\ot 3}$
\[
 \cyb(x,y)=[x^{1,2},y^{1,3}]+[x^{1,3},y^{2,3}]+[x^{1,2},y^{2,3}]+[y^{1,2},x^{1,3}]+[y^{1,3},x^{2,3}]+[y^{1,2},x^{2,3}]
\]
and define $\cyb(x)=\frac12 \cyb(x,x)$.
\begin{lem}\label{lem:diagCYB}
There exists a bilinear map
\[
\bcyb:(\wedge^2 \mf m)^{\mf h}\times (\wedge^2 \mf m)^{\mf h}\rightarrow (\wedge^3 \mf m)^{\mf h}
\]
fitting in the following diagram
\begin{center}
 \begin{tikzpicture}
\matrix (m) [matrix of math nodes, row sep=5em,
column sep=3em, text height=1.5ex, text depth=0.25ex]
{\wedge^2 \mf g \times \wedge^2 \mf g & (\wedge^2 \mf g)^{\mf h}\times (\wedge^2 \mf g)^{\mf h}& (\wedge^2 \mf m)^{\mf h}\times (\wedge^2 \mf m)^{\mf h}\\
\wedge^3 \mf g&(\wedge^3 \mf g)^{\mf h}&(\wedge^3 \mf m)^{\mf h}\\} ;
\path[left hook->,font=\scriptsize]
(m-1-2) edge node[auto] {} (m-1-1)
(m-2-2) edge node[auto] {} (m-2-1);
\path[->>,font=\scriptsize]
(m-1-2) edge node[below] {} (m-1-3)
(m-2-2) edge node[below] {} (m-2-3)
;
\path[<-,font=\scriptsize]
(m-2-1) edge node[above,sloped] {$\cyb$} (m-1-1)
(m-2-2) edge node[above,sloped] {$\cyb$} (m-1-2)
(m-2-3) edge node[above,sloped] {$\bcyb$} (m-1-3)
;
\end{tikzpicture}
\end{center}
\end{lem}
\begin{proof}
The commutativity of the left square is clear.

$\cyb$ maps $(\wedge^2 \mf g)^{\mf h}\times (\mf h\wedge \mf g)^{\mf h}$ to $(\mf h \wedge \bigwedge^2 \mf g)^{\mf h}$, meaning that $\cyb$ induces a bilinear map
\[
\frac{(\wedge^2 \mf g)^{\mf h}}{(\mf h \wedge \mf g)^{\mf h}}\times \frac{(\wedge^2 \mf g)^{\mf h}}{(\mf h \wedge \mf g)^{\mf h}} \rightarrow \frac{(\wedge^3 \mf g)^{\mf h}}{(\mf h \wedge \bigwedge^2 \mf g)^{\mf h}}
\]
As the decomposition $\mf g=\mf h \oplus \mf m$ is reductive, 
\begin{equation}\label{eq:quoCyb}
\wedge^2\mf g\cong\wedge^2\mf h \oplus (\mf h \ot \mf m) \oplus \wedge^2 \mf m.
\end{equation}

Hence, 
\[
\frac{(\wedge^2 \mf g)^{\mf h}}{(\mf h \wedge \mf g)^{\mf h}}\cong(\wedge^2 \mf m)^{\mf h}
\]
In the same way,
\[
\frac{(\wedge^3 \mf g)^{\mf h}}{(\mf h \wedge \bigwedge^2 \mf g)^{\mf h}}\cong(\wedge^3 \mf m)^{\mf h}
\]
Composing~\eqref{eq:quoCyb} with these isomorphisms, we obtain a map
\[
\bcyb:(\wedge^2 \mf m)^{\mf h}\times (\wedge^2 \mf m)^{\mf h}\longrightarrow (\wedge^3 \mf m)^{\mf h}
\]
which makes the right diagram commutes.
\end{proof}
\begin{proof}[Proof of Prop.~\ref{prop:main}] 
Let
\[
 \Psi= \sum_{k\geq 0} \hbar^k \Psi_k,\ \Psi_k \in U(\mf h) \ot U(\mf g)^{\ot 2}
\]
and
\[
 \Psi'= \sum_{k\geq 0} \hbar^k \Psi'_k,\ \Psi'_k \in U(\mf h) \ot U(\mf g)^{\ot 2}
\]
Assume that there exist $n>1$ such that for all $1\leq k < n$ $\Psi_k=\Psi'_k$. The mixed pentagon equation~\eqref{eq:mixedPenta} implies that $d_1^{2,3}(\Psi_n-\Psi_n')=0$. Hence, there exists $a \in (\wedge^2 \mf m)^{\mf h}$ and $\mf g \in (U(\mf h) \ot U(\mf g))^{\mf h}$ such that
\[
\Psi_n-\Psi'_n=1\ot a +d_1^{1,2}(g)
\]
Both $\Psi_n$ and $\Psi_n'$ belongs to $U(\mf h)_{\leq n-1}\ot U(\mf g)^{\ot 2}$ by assumption, then so does $d_1^{1,2}( g)$. Hence, according to Theorem~\ref{thm:damien}, $g$ can be chosen is such a way that it belongs to $U(\mf h)_{\leq n}\ot U(\mf g)$. Moreover, the fact that $d_1^{1,2}(g)\in U(\mf h)_{\leq n-1}\ot U(\mf g)^{\ot 2}$ implies that the $U(\mf h)$-degree $n$ part $g_n$ of $g$ satisfies
\[
g_n^{1,23}-g_n^{1,2}-g^{1,3}=0
\]
because $g_n^{12,3}=g_n^{1,3}+$\{lower degree terms w.r.t. the $U(\mf h)$ component\}. Hence, $g_n$ actually belongs to $U(\mf h)\ot  g$, meaning that $G=\exp(\h^n g) \in \cg$. Thus, let
\[
\Psi''=G\star \Psi'.
\]
By construction, $\Psi-\Psi'' \in \h^{n-1}U(\mf h)\ot U(\mf g)^{\ot 2}$ and
\begin{equation}\label{eq:obstruction}
\Psi_n-\Psi''_n=1\ot a
\end{equation}

Set
\[
\bar r= \alt_2(\rho_0)=2\rho_0 \in  (\wedge^2 \mf m)^{\mf h}
\]

\begin{lem}
\label{lem:cyb}

\[
 \bcyb(\overline{r},a)=0 \in (\wedge^3 \mf m)^{\mf h}
\]
\end{lem}
\begin{proof}
Expanding the mixed pentagon equation up to order $n+1$ leads to:
\[
 d_1^{2,3}(\Psi_{n+1}-\Psi''_{n+1}) = \rho_0^{2,34}a^{3,4} + a^{2,34}\rho_0^{3,4} - \rho_0^{23,4}a^{2,3}- a^{23,4}\rho_0^{2,3} \in U(\mf h)\ot U(\mf g)^{\ot 3}
\]

Set $\psi_{n+1}=(\epsilon \ot \alt_3)( d(\Psi_{n+1}-\Psi''_{n+1}))$ we have:
\[
 \alt_3(\rho_0^{1,23}a^{2,3} + a^{1,23}\rho_0^{2,3} - \rho_0^{12,3}a^{1,2}- a^{12,3}\rho_0^{1,2})=\psi_{n+1} 
\]

 Each component of $a$ and $\rho_0$ are primitive, that is
 \begin{align}
  a^{12,3}&=a^{1,3}+a^{2,3}\\
  a^{1,23}&=a^{1,2}+a^{1,3}\\
  \rho_0^{12,3}&=\rho_0^{1,3}+\rho_0^{2,3}\\
  \rho_0^{1,23}&=\rho_0^{1,2}+\rho_0^{1,3}
 \end{align}
It follows that:
\begin{align*}
 &\rho_0^{1,2}a^{2,3}+\rho_0^{1,3}a^{2,3} + a^{1,2}\rho_0^{2,3}+a^{1,3}\rho_0^{2,3} - \rho_0^{1,3}a^{1,2}-\rho_0^{2,3}a^{1,2}- a^{1,3}\rho_0^{1,2}-a^{2,3}\rho_0^{1,2} \\
 - &\rho_0^{2,1}a^{1,3}-\rho_0^{2,3}a^{1,3} - a^{2,1}\rho_0^{1,3}-a^{2,3}\rho_0^{1,3} + \rho_0^{2,3}a^{2,1}+\rho_0^{1,3}a^{2,1}+ a^{2,3}\rho_0^{2,1}+a^{1,3}\rho_0^{2,1} \\
 -&\rho_0^{3,2}a^{2,1}-\rho_0^{3,1}a^{2,1} - a^{3,2}\rho_0^{2,1}-a^{3,1}\rho_0^{2,1} + \rho_0^{3,1}a^{3,2}+\rho_0^{2,1}a^{3,2}+ a^{3,1}\rho_0^{3,2}+a^{2,1}\rho_0^{3,2} \\
 - &\rho_0^{1,3}a^{3,2}-\rho_0^{1,2}a^{3,2} - a^{1,3}\rho_0^{3,2}-a^{1,2}\rho_0^{3,2} + \rho_0^{1,2}a^{1,3}+\rho_0^{3,2}a^{1,3}+ a^{1,2}\rho_0^{1,3}+a^{3,2}\rho_0^{1,3} \\
 +&\rho_0^{2,3}a^{3,1}+\rho_0^{2,1}a^{3,1} + a^{2,3}\rho_0^{3,1}+a^{2,1}\rho_0^{3,1} - \rho_0^{2,1}a^{2,3}-\rho_0^{3,1}a^{2,3}- a^{2,1}\rho_0^{2,3}-a^{3,1}\rho_0^{2,3} \\
 +&\rho_0^{3,1}a^{1,2}+\rho_0^{3,2}a^{1,2} + a^{3,1}\rho_0^{1,2}+a^{3,2}\rho_0^{1,2} - \rho_0^{3,2}a^{3,1}-\rho_0^{1,2}a^{3,1}- a^{3,2}\rho_0^{3,1}-a^{1,2}\rho_0^{3,1} \\
 =&\psi_{n+1}
\end{align*}
As $a^{2,1}=-a^{1,2}$, it implies that:
\begin{align*}
\psi_{n+1}&=[\rho_0^{1,2},a^{2,3}]+[\rho_0^{1,3},a^{2,3}] + [a^{1,2},\rho_0^{2,3}]+[a^{1,3},\rho_0^{2,3}] - [\rho_0^{1,3},a^{1,2}] - [a^{1,3},\rho_0^{1,2}]\\
  &-[\rho_0^{2,1},a^{1,3}] +[a^{2,3},\rho_0^{2,1}] + [\rho_0^{3,2},a^{1,2}]+[\rho_0^{3,1},a^{1,2}] - [\rho_0^{3,1},a^{2,3}] - [a^{1,3},\rho_0^{3,2}]
\end{align*}
that is:
\begin{align*}
 \psi_{n+1}=&[\rho_0^{1,2}-\rho_0^{2,1},a^{2,3}]+[\rho_0^{1,3}-\rho_0^{3,1},a^{2,3}] + [a^{1,2},\rho_0^{2,3}-\rho_0^{3,2}]\\
 +&[a^{1,3},\rho_0^{2,3}-\rho_0^{3,2}] + [a^{1,2},\rho_0^{1,3}-\rho_0^{3,1}] + [\rho_0^{1,2}-\rho_0^{2,1},a^{1,3}]\\
 =&\cyb(\overline{r},a)
\end{align*}
because $\overline{r}=\rho_0^{1,2}-\rho_0^{2,1}$. The projection of $\psi_{n+1}$ in $(\wedge^3 \mf m)^{\mf h}$ is 0 by Corollary~\ref{cor:cobord}. Hence, as both $\bar r$ and $a$ belong to $(\wedge^2 \mf m)^{\mf h}$, one has:
\[
P^{\ot 3}(\cyb(\bar r,a))=\bcyb(\bar r,a)=0
\]
by Lemma~\ref{lem:diagCYB}.
\end{proof}

Let $J\in \widehat{S(\mf h)}\ot U(\mf g)^{\ot 2}$ be the image of $\frac{\Psi-1}{\h}$ through the tensor product of the reduction maps $U(\mf h)[[\h]]' \rightarrow \widehat{S(\mf h)}$ and $U(\mf g)[[\h]] \rightarrow U(\mf g)$. The coproduct of $U(\mf h)[[\h]]$ induces an algebra map
\[
\Delta:\widehat{S(\mf h)}[[\h]]\rightarrow \widehat{S(\mf h)}\ot U(\mf g)[[\h]]
\]
defined by
\[
u_i \longmapsto u_i\ot 1+1\ot \h h_i
\]
The reduction of the mixed pentagon equation implies that $J$ satisfies
\[
J^{1,3,4}+J^{1,2,34}-J^{1,23,4}-J^{1,2,3}=0
\]
because the reduction of $\frac{1\ot \Phi_{\kz}-1}{\h}$ in $\widehat{S(\mf h)} \ot U(\mf g)^{\ot 3}$ is 0, and using the fact that
\[
J^{12,3,4}=J^{1,3,4}+O(\h).
\]
It means that $(\id \ot b^{2,3})(J)=0$ where $b$ is the differential of the coHochschild cochain complex $\bigoplus_{n\geq 0} U(\mf g)^{\ot n}$ defined by the same formula as in~\eqref{eq:coHochschild}. According to~\cite[Prop.~2.2]{Drinfeld1990}, the linear operator $\alt_n:U(\mf g)^{\ot n}\rightarrow U(\mf g)^{\ot n}$ maps cocycles into $\wedge^n \mf g$. Hence,
\[
\rho=(\id \ot \alt_2)(J) \in \widehat{S(\mf h)} \ot \wedge^2 \mf g
\]
Let $Z=\alt_3(\frac{\Phi_{\kz}-1}{\h^2}) \mod \h$ (recall~\cite{Drinfeld1990} that $\Phi_{\kz}=1+O(\h^2)$). A direct consequence of the mixed pentagon equation is the following~\cite{Enriquez2003,Xu2002}:
\begin{prop}
$\rho$ satisfies the modified Dynamical Yang-Baxter Equation (mCDYBE)
\begin{equation}
\label{eq:DCYBE}
(\id \ot  \cyb)(\rho) + (\id\ot \widetilde\alt)(d\rho) =1\ot Z \in \widehat{S(\mf h)}\ot (\wedge^3 \mf g)^{\mf g}
\end{equation}
where $d:\widehat{S(\mf h)}\ot \mf g^{\ot 2}\rightarrow \widehat{S(\mf h)} \ot \mf h \ot \mf g^{\ot 2}$ is the formal de~Rham differential and 
\[
\widetilde\alt(x)=x^{1,2,3}-x^{2,1,3}+x^{3,1,2}
\]
\end{prop}
As a consequence~\cite{Etingof2001e}, $\bar r$ satisfies the modified Classical Yang-Baxter Equation
\begin{equation}
\label{eq:CYBE}
 \bcyb(\bar{r})=\overline{Z} \in \wedge^3 \mf m
\end{equation}
where $x\mapsto \overline{x}$ is the projection $(\wedge^n \mf g)^{\mf h} \rightarrow (\wedge^n \mf m)^{\mf h}$. Let $\varepsilon$ be a formal variable, we have the following:
\begin{lem}
$\overline{r}+\varepsilon a$ is an infinitesimal deformation of $\overline{r}$, i.e. 
 \[
  \bcyb(\overline{r}+\varepsilon a)=\bar{Z}  \in (\wedge^3 \mf m)^{\mf h}[\varepsilon]/(\varepsilon^2)
 \]
\end{lem}
\begin{proof}
 As
\[
 \bcyb(r+\varepsilon  a)=\bcyb(r)+2\varepsilon \bcyb(r,a) \mod \varepsilon^2
\]
it follows from lemma~\ref{lem:cyb} and from the fact that
\[
 \bcyb(\overline r)=\overline{Z}
\]
that $\bcyb(\overline r+\varepsilon a)=\overline{Z} \mod \varepsilon^2$.
\end{proof}

Solution of the mCDYBE in the (semi-)simple case were classified by Etingof--Varchenko~\cite[Theorem 3.10]{Etingof1998b}. It implies that $\rho$ is actually the Taylor expansion around the origin of a holomorphic function
\[
r:D \rightarrow (\wedge^2 \mf g)^{\mf h}
\]
where $D \subset \mf h^*$ is a neighbourhood of 0. Moreover, if
\[
r(\lambda)=r_{\mf h}(\lambda) +\sum_{\alpha \in R} \phi_{\alpha}(\lambda) e_{\alpha} \ot e_{-\alpha}
\]
the assumption (b) of Theorem~\ref{thm:shift} means that $\phi_{\alpha}(0) \neq \pm 1$ which implies that there exists $\nu \in \mf h^*$ such that
\[
\forall \alpha \in R,\ \phi_{\alpha}(\lambda)= \coth(2(\alpha,\lambda-\nu))
\]
An infinitesimal shift of $r(\lambda)$ by an element $\mu$ of $\mf h^*$ induces an infinitesimal deformations of it as a solution of the mCDYBE:
\[
r(\lambda + \varepsilon \mu)=r(\lambda)+\varepsilon(\mu \ot \id^{\ot 2})(dr(\lambda))+O(\varepsilon^2)
\]
Evaluating the above expression in 0 and projecting in $(\wedge^2 \mf m)^{\mf h}$, it follows that $\mu$ also induces an infinitesimal deformation of $\overline{r}$ as a solution of the (modified) Classical Yang-Baxter equation in $(\wedge^2 \mf m)^{\mf h}$. This define a linear map $\theta:\mf h^* \rightarrow \mathcal T_{def}(\bar r)$ by
\[
\mu\mapsto \overline{(\mu \ot \id\ot \id)(dr(0))}
\]
where $\mathcal T_{def}(\bar r)$ is the vector space\footnote{It is a vector space because the map $a \mapsto \cyb(\overline r,a)$ is a linear map.} of infinitesimal deformation of $\bar{r}$, that is:
\[
 \mathcal T_{def}(\bar r)=\{x \in (\wedge^2 \mf m)^{\mf h}\ \vert\ \bcyb(\bar{r},x)=0 \}
\] 
\begin{prop}
The linear map $\theta$ is an isomorphism.
\end{prop}
\begin{proof}
Let us first prove that $\dim \mathcal T_{def}(\bar r)\leq \dim \mf h^*$. For $\alpha \in R$, let $r_{\alpha}=\phi_{\alpha}(0)$, by construction
\[
\bar r=\sum_{\alpha \in R} r_{\alpha} e_{\alpha}\ot e_{-\alpha}
\]
and $r_{-\alpha}=-r_{\alpha}$ by antisymmetry. The only nontrivial parts of $\bcyb(\bar r,a)$ are those belonging to $\mf g_{\alpha} \ot\mf g_{\beta} \ot\mf g_{\gamma}$ with $\alpha+\beta+\gamma=0$, because $[e_{\alpha},e_{-\alpha}]=0$ in $\mf m=\mf g/\mf h$. 

Write
\[
a=\sum_{\alpha\in R}a_{\alpha}e_{\alpha}\ot e_{-\alpha}
\]
with $a_{-\alpha}=-a_{\alpha}$. The fact that $\bcyb(\bar r,a)=0$ leads to the following relation:
\begin{equation}\label{eq:DefoCYB}
 a_{\alpha+\beta} (r_{\alpha} +r_{\beta})=a_{\alpha} (r_{\beta} - r_{\alpha+\beta}) + a_{\beta} (r_{\alpha} - r_{\alpha+\beta})
\end{equation}

Following again~\cite{Etingof1998b}, the mCDYBE for $r(\lambda)$ implies the following relation in $\mf g_{\alpha} \ot\mf g_{\beta} \ot\mf g_{-\alpha-\beta}$, $\forall \alpha,\beta\neq 0, \alpha+\beta \in R$:
\[
r_{\alpha}r_{\beta}-r_{\alpha}r_{\alpha+\beta}-r_{\beta}r_{\alpha+\beta}+1=0
\]
Hence, if $r_{\alpha}+r_{\beta}$ were equal to 0, it would imply that $r_{\alpha}^2=1$, hence that $r_{\alpha}=\pm 1$ contradicting assumption (b) of Theorem~\ref{thm:shift}. Thus, if $a_{\alpha}=0$ for any $\alpha \in \Pi$, then $a_{\alpha}=0$ for all $\alpha \in R$. It means that the linear map
\[
\mathcal T_{def} \rightarrow \C^{\dim \mf h^*}
\]
defined by
\[
(a_{\alpha})_{\alpha \in R^+} \longmapsto (a_{\alpha})_{\alpha \in \Pi}
\]
is injective. Hence, $\dim \mathcal T_{def}(\bar r)\leq\dim \mf h^*$.

Observe now that for $\alpha \in R$, $d\phi_{\alpha}(0)$ is the linear map $\mf h^* \rightarrow \C$ defined by 
\[
\mu \longmapsto - 2(\alpha,\mu)\csch^2(2(\alpha,\nu)) 
\]
where $\csch(z)=\frac2{e^z-e^{-z}}$ is the hyperbolic cosecant. Hence, let 
\[
\bar r(\lambda)=\sum_{\alpha \in R}\phi_{\alpha}(\lambda)e_{\alpha}\ot e_{-\alpha}
\]
be the function $D\rightarrow (\wedge^2 \mf m)^{\mf h}$ induced by $r(\lambda)$. Let $\mu_0 \in \mf h^*$ be such that $(\mu_0 \ot \id^{\ot 2})(d\bar r(0))=0$. As $\csch^2(z)\neq 0$ for all $z\in \C$, it implies that
\[
\forall \alpha \in R^+,\ (\alpha,\mu_0)=0
\]
and thus that $\mu_0=0$ because $(\ ,\ )$ is non-degenerate. It follows that the map $\theta$ is injective, hence an isomorphism.
\end{proof}
Hence, let $\mu_a \in \mf h^*$ be such that $\theta(\mu_a)=a$ and set $\tilde\mu_a =\h^{n-2} \mu_a \in \mf h^*[[\h]]$. By definition, 
\[
(\id\ot \alt_2)(\Psi_2)=dr(0)+1\ot U(\mf g)^{\ot 2}
\]
Hence, $(\Psi_{\tilde\mu_a})_k=(\Psi)_k$ for $k<n$ and as shifts act trivially on $1\ot U(\mf g)^{\ot 2}$, 
\[
\alt_2\circ (\epsilon \ot P^{\ot 2})(\Psi_{\tilde\mu_a})_n = \alt_2\circ (\epsilon \ot P^{\ot 2})(\Psi_n)+a
\]
Using equation~\eqref{eq:obstruction}, it means that 
\[
\alt_2\circ (\epsilon \ot P^{\ot 2})((\Psi_{\tilde\mu_a})_n-\Psi''_n)=0
\]
and thus that $(\Psi_{\tilde\mu_a})_n$ and $\Psi'_n$ are cohomologous. Hence, there exists $g \in (U(\mf h)\ot U(\mf g))^{\mf h}$ such that
\[
 \Psi'_n=(\Psi_{\tilde\mu_a})_n +d_1(g)
\]
One shows as before that $G=\exp(\h^ng) \in \cg$, and by construction
\[
\Psi'=G\star \Psi_{\tilde\mu_a} \mod \h^{n+1}
\]
The Proposition is proved.
\end{proof}
The proof of the Theorem then follows from an induction on $n$.
\end{proof}
\subsection{Equivalence}

\begin{prop}
\label{prop:rigid}
 If moreover $\Psi,\Psi'$ satisfy the octagon equation with $E_{\kz,\sigma}$, then they are actually twist equivalents.
\end{prop}
\begin{proof}
 According to theorem~\ref{thm:shift}, we can assume that there exists $\mu \in \mf h^*[[\h]]$ such that 
\[
\Psi=\Psi'_{\mu}
\]
Let $E_{\mu}$ be the image of $E_{\kz,\sigma}$ by the shift by $\mu$. Hence, $E_{\mu}=e^{\h \check \mu^{(2)}}E_{\kz}$ where $\check \mu=(\mu \ot \id)(t_{\mf h})\in \mf h[[\h]]$. Then, on the one hand, $\Psi$ satisfies the octagon equation with $E_{\kz}$ by definition. On the other hand, as $\Psi=\Psi'_{\mu}$ and because $\Psi'_{\mu}$ satisfies the octagon equation with $E_{\mu}$, $\Psi$ has to satisfies the octagon equation with $E_{\mu}$ too. Hence,
\[
\Psi^{-1}e^{-\h t^{2,3}/2} \Psi^{1,3,2} =(E_{\kz,\sigma}^{12,3})^{-1}\Psi^{-1}e^{\h t^{2,3}/2} \Psi^{1,3,2}E_{\kz,\sigma}^{1,3}
\]
and
\[
\Psi^{-1}e^{-\h t^{2,3}/2} \Psi^{1,3,2} =(E_{\mu}^{12,3})^{-1}\Psi^{-1}e^{\h t^{2,3}/2} \Psi^{1,3,2}E_{\mu}^{1,3}
\]

Therefore, the right hand sides of these equations are equal, meaning that $e^{\h \check\mu}$ satisfies:
\[
 \Psi^{-1} e^{\h t^{2,3}/2} \Psi^{1,3,2}= e^{\h \check \mu^{(3)}} \Psi^{-1} e^{\h t^{2,3}/2} \Psi^{1,3,2} e^{-\h \check \mu^{(3)}}
\]
which implies that
\[
 [\check\mu^{(3)}, \Psi_1^{1,3,2}-\Psi_1+1\ot t/2]=0
\]
that is
\[
 [\check\mu^{(3)}, \Psi_1^{1,3,2}-\Psi_1+1\ot t_{\mm}/2]=0
\]
because
\[
 [\check\mu^{(3)}, t_{\mf h}]=0
\]
Write
\[
\Psi_1^{1,3,2}-\Psi_1+1\ot t/2=1\ot\sum_{\alpha \in R} \lambda_{\alpha} e_{\alpha} \ot e_{-\alpha}
\]
and recall that
\[
 t_{\mm}/2=\sum_{\alpha \in R^+} (e_{\alpha} \ot e_{-\alpha}+e_{-\alpha} \ot e_{\alpha})
\]
Hence, as $\Psi_1^{1,3,2}-\Psi_1$ is antisymmetric, $\lambda_{\alpha}+\lambda_{-\alpha}=1$. It means that at least one of $\lambda_{\alpha},\lambda_{-\alpha}$ is non-zero. As
\[
 [\check\mu^{(3)}, \Psi_1^{1,3,2}-\Psi_1+1\ot t_{\mm}/2] =1\ot \sum_{\alpha \in R} -\lambda_{\alpha}\alpha(\check \mu)e_{\alpha} \ot e_{-\alpha}
\]
it means that $\alpha(\check\mu)=0$ for all $\alpha \in R$, and thus that $\check\mu=0$. Therefore, $\mu=0$.
\end{proof}

Let us now check that the dynamical pseudo-twists $\Psi,\Psi'$ associated to the QRAs $B_{\kz},B$ match the assumptions of Theorem~\ref{thm:shift}. The quasi-classical limit of $\Psi_{\kz,\sigma}$ was computed in~\cite{Enriquez2005}, and the quasi-classical limit of $\Psi_{\h}$ can be computed by considering the quasi-classical limit of the modified ABRR equation. Actually, these two results can be deduced from the following more general result:
\begin{lem}\label{lem:s.c.}
Let $\Psi \in 1+\h (U(\mf h)[[\h]]' \hat \ot U(\mf g)^{\ot 2}[[\h]])^{\mf h}$ be a solution of the mixed pentagon equation which satisfies the octagon equation with $E=E_{\kz,\sigma}$. Then there exists $G \in \cg$ such that
\[
\frac{G\star \Psi-1}{\h} \mod \h=\frac12\left(\id \ot \frac{\id+\sigma}{\id-\sigma} \right )((t-\tth)/2) \in \wedge^2 \mf m
\]
\end{lem}
\begin{proof}
The quasi-classical limit of the octagon equation is
\[
\tilde \sigma^{(3)}(\tth^{1,3}+\tth^{2,3}+\frac12 \tth^{3,3})=(\rho_0^{1,3,2}-\rho_0+t^{2,3}/2)\tilde\sigma^{(3)}+\tth^{1,3}+\frac12\tth^{3,3}+\tilde\sigma^{(3)}(\rho_0-\rho_0^{1,3,2}+t^{2,3}/2)
\]
where 
\[
\rho_0:=\frac{\Psi-1}{\h} \mod \h
\]
Hence,
\begin{align*}
 \tilde \sigma^{(3)}\tth^{2,3}&=(\rho_0^{1,3,2}-\rho_0+t^{2,3}/2)\tilde\sigma^{(3)}+\tilde\sigma^{(3)}(\rho_0-\rho_0^{1,3,2}+t^{2,3}/2)\\
\tth^{2,3}&=(\id - \sigma^{(3)})(\rho_0-\rho_0^{1,3,2})+(\id + \sigma^{(3)})(t^{2,3}/2)
\end{align*}
Finally
\begin{align*}
\rho_0-\rho_0^{1,3,2}&= (\id -\sigma^{(3)})^{-1}((\id + \sigma^{(3)})(t^{2,3}/2)-\tth^{2,3})\\
&= (\id -\sigma^{(3)})^{-1}((\id + \sigma^{(3)})((t^{2,3}-\tth^{2,3})/2)+\tth^{2,3}-\tth^{2,3})\\
&= \frac{\id + \sigma^{(3)}}{\id -\sigma^{(3)}}(t_{\mm}^{2,3}/2)
\end{align*}
as $(\id + \sigma^{(2)})(\tth/2)=\tth$. Thus, 
\[
r_0:=(\epsilon \ot \alt_2)(\rho_0)=\frac{\id + \sigma^{(2)}}{\id -\sigma^{(2)}}(t_{\mm}/2)=\frac{\id + \sigma^{(2)}}{\id -\sigma^{(2)}}(\sum_{\alpha\in R}e_{\alpha}\ot e_{-\alpha}+e_{-\alpha}\ot e_{\alpha})
\]
The fact that $r_0\in (\wedge^2 \mf m)^{\mf h}$ can be checked directly, but follows more generally from the fact that $\rho_0$ is a cocycle in $C_1$, as implied by the mixed pentagon equation. Then according to Theorem~\ref{thm:damien} it means that there exists $g \in U(\mf h)_{\leq 1}\ot U(\mf g)$ such that
\[
\rho_0=\frac12r_0-d_1^{1,2}(g)
\]
Therefore, $G=\exp(\h g) \in \cg$ and
\[
\frac{G\star \Psi-1}{\h} \mod \h=\frac12r_0
\]
as required.
\end{proof}
\begin{proof}[Proof of Theorem~\ref{thm:main}]
Lemma~\ref{lem:s.c.} implies that there exists $G,G' \in \cg$ such that $\tilde \Psi=G\star \Psi$ and $\tilde\Psi'=G'\star \Psi'$ satisfy the assumptions of Theorem~\ref{thm:shift}.  Consequently, there exists $(\tilde G,\mu)\in \cg \rtimes \mf h^*[[\h]]$ such that 
\[
\tilde\Psi=\tilde G \star \tilde \Psi'_{\mu}
\]
Note that $(U(\mf h) \ot U(\mf g)[[\hbar]])^{\mf h}=U(\mf h) \ot (U(\mf g)^{\mf h})[[\hbar]]$, which implies that $\mf h$-invariant twists actually commutes with $E_{\kz,\sigma}$.

It follows that $\tilde\Psi$ and $\tilde G\star\tilde\Psi'$ both satisfy the octagon equation with $E_{\kz,\sigma}$. Then, Proposition~\ref{prop:rigid} implies that $\mu=0$, and hence that they are equal. Finally
\[
\Psi=(G^{-1}\tilde G G)\star \Psi'
\]
implying that the QRAs $B$ and $B_{\kz}$ are twist-equivalent. The Theorem is proved.
\end{proof}


\end{document}